\documentclass{article}

\usepackage{xparse}
\usepackage{amsmath}
\usepackage{amssymb}
\usepackage{enumitem}
\usepackage{mathrsfs}
\usepackage{fancyhdr}
\usepackage{pgf}
\usepackage{array}
\usepackage{longtable}
\usepackage{graphicx}

\title{On $p$-adic incomplete Mellin transforms and $p$-adic incomplete gamma-functions}
\author{Paul Buckingham}
\date{}
\numberwithin{equation}{section}
\numberwithin{table}{section}
\newenvironment{proof}{\noindent\emph{Proof}.}{%
\hspace{\stretch{1}}\rule{1.5ex}{1.5ex} \vspace{5 mm}}

\allowdisplaybreaks

\newtheorem{theorem}{Theorem}[section]
\newtheorem{prop}[theorem]{Proposition}
\newtheorem{lemma}[theorem]{Lemma}
\newtheorem{cor}[theorem]{Corollary}

\newtheorem{examp}[theorem]{Example}
\newcommand{\bex}{\begin{examp}\normalfont}
\newcommand{\eex}{\end{examp}}
\newtheorem{remark}[theorem]{Remark}
\newcommand{\brem}{\begin{remark}\normalfont}
\newcommand{\erem}{\end{remark}}

\makeatletter
\newcommand\goinvisible{\pgfsys@begininvisible}
\newcommand\govisible{\pgfsys@endinvisible}
\makeatother

\newcommand{\dsh}{---}
\newcommand{\gammahat}{\dot{\Gamma}}
\newcommand{\extexp}{\mathscr{E}} 
\newcommand{\convderiv}{D}
\newcommand{\amice}{\mathcal{A}}
\NewDocumentCommand{\dirac}{o}{\lambda\IfValueTF{#1}{_{#1}}{}} 
\newcommand{\prodcorr}{\Theta}
\newcommand{\actcorr}{\Lambda}
\newcommand{\decayegf}{\Cp\pwr{\gfvar}_0}
\newcommand{\valueegf}{\Cp\pwr{t}_{\mathrm{cts}}}
\NewDocumentCommand{\twovar}{mo}{\mathcal{T}\IfValueTF{#2}{_{#2}}{}^{#1}}
\newcommand{\resch}[1]{\operatorname{ch}(#1)} 
\NewDocumentCommand{\princof}{O{\p}m}{\langle#2\rangle_{#1}}
\NewDocumentCommand{\trable}{o}{\mathscr{F}\IfValueTF{#1}{_{#1}}{}} 
\newcommand{\gexp}{\operatorname{Exp}} 
\newcommand{\gentmap}{\mathscr{I}}
\NewDocumentCommand{\gent}{sm}{
  \gentmap
  \IfBooleanTF{#1}{%
    \mathopen{}\left(#2\right)\mathclose{}%
  }{%
    (#2)%
  }%
}
\newcommand{\pret}{B} 
\newcommand{\postt}{C} 

\NewDocumentCommand{\Fset}{O{p}}{\mathscr{F}_{#1}}

\newcommand{\oop}{\big(\tfrac{1}{p}\big)}

\newcommand{\pmspacechar}{V}
\NewDocumentCommand{\pmspace}{O{p}m}{\pmspacechar_{#1}^{#2}}
\NewDocumentCommand{\pmspacecup}{O{p}}{\pmspacechar_{#1}}
\NewDocumentCommand{\pmspacez}{O{p}m}{\pmspacechar_{#1,0}^{#2}}
\NewDocumentCommand{\pmspacecupz}{O{p}}{\pmspacechar_{#1,0}}

\newcommand{\varholder}{{\,\cdot\,}}
\NewDocumentCommand{\pgamma}{O{\varholder}m}{\Gamma_p(#1,#2)}
\NewDocumentCommand{\lball}{mmo}{B_{< #2}(#1\IfValueTF{#3}{,#3}{})}

\newcommand{\binf}[1]{{\textstyle{\ch{\cdot}{#1}}}}
\newcommand{\myst}{\ast}
\newcommand{\xtox}{{\scriptstyle{\mathfrak{X}}}}

\newcommand{\gfvar}{t}

\newcommand{\gfn}{\widehat{\Gamma}}
\newcommand{\lgfn}{\widehat{\gamma}}
\newcommand{\inttmap}{\mathscr{I}}
\NewDocumentCommand{\inttrans}{sm}{%
  \inttmap
  \IfBooleanTF{#1}{%
    \mathopen{}\left(#2\right)\mathclose{}%
  }{%
    (#2)%
  }%
}

\newcommand{\lan}[2]{\mathrm{la}(#1,#2)}

\newcommand{\qfsymb}{q}
\newcommand{\qf}{\qfsymb}

\NewDocumentCommand{\pint}{o}{\mathbf{A}\IfValueTF{#1}{_{#1}}{}}
\NewDocumentCommand{\maxid}{o}{\mathbf{M}\IfValueTF{#1}{_{#1}}{}}
\NewDocumentCommand{\units}{o}{\mathbf{U}\IfValueTF{#1}{_{#1}}{}}
\NewDocumentCommand{\punits}{o}{\mathbf{U}\IfValueTF{#1}{_{#1}}{}^{(1)}}
\NewDocumentCommand{\expdom}{o}{\mathbf{E}\IfValueTF{#1}{_{#1}}{}}

\newcommand{\ra}{\rightarrow}

\newcommand{\bb}[1]{\mathbb{#1}}

\newcommand{\Z}{\bb{Z}}
\newcommand{\Q}{\bb{Q}}
\newcommand{\R}{\bb{R}}
\newcommand{\C}{\bb{C}}
\newcommand{\Zp}{\Z_p}

\newcommand{\Qp}{\Q_p}

\newcommand{\F}{\mathbb{F}}
\NewDocumentCommand{\Fp}{O{p}}{\F_{#1}}

\NewDocumentCommand{\Fpbr}{O{p}}{\br{\F}_{#1}}

\newcommand{\p}{\mathfrak{p}}

\newcommand{\q}{\mathfrak{q}}

\newcommand{\of}{\circ}

\newcommand{\br}[1]{\bar{#1}} 

\newcommand{\st}{^\times}
\providecommand{\ie}{i.e., }

\newcommand{\setm}{\! \smallsetminus \!}
\newcommand{\lcm}{\mathrm{lcm}}

\newcommand{\oh}{\mathcal{O}}

\newcommand{\spc}{\mbox{ }}

\newcommand{\Mod}{\spc\mathrm{mod}\spc}

\newcommand{\con}{\subseteq}

\newcommand{\eps}{\epsilon}

\newcommand{\sat}{\ |\ }
\newcommand{\SAT}{\ \Bigg|\ }

\newcommand{\id}{\operatorname{id}}
\newcommand{\Cp}{\mathbb{C}_p}

\newcommand{\teich}{\omega}

\newcommand{\calL}{\mathcal{L}}

\newcommand{\pwr}[1]{[\![#1]\!]}

\newcommand{\laurr}[1]{\{\!\{X\}\!\}}

\newcommand{\ch}[2]{{#1 \choose #2}}
\newcommand{\frakm}{\mathfrak{m}}

\newcommand{\beq}{\begin{equation}}
\newcommand{\eeq}{\end{equation}}

\newcommand{\bea}{\begin{eqnarray*}}
\newcommand{\eea}{\end{eqnarray*}}
\newcommand{\beal}{\begin{eqnarray}}
\newcommand{\eeal}{\end{eqnarray}}
\newcommand{\bcs}{\left\{\begin{array}{ll}}
\newcommand{\ecs}{\end{array}\right.}

\newcommand{\one}{\mathbf{1}}

\newcommand{\calS}{\mathcal{S}}

\newcommand{\calK}{\mathcal{K}}

\newcommand{\rl}{\operatorname{Re}}
\newcommand{\ig}{\operatorname{Im}}

\newcommand{\eqcom}[1]{\quad\text{#1}}
\NewDocumentCommand{\eqcomb}{mo}{\quad\text{(#1)\IfValueTF{#2}{#2}{}}}
\NewDocumentCommand{\mtr}{om}{%
  \IfValueTF{#1}{%
    \left(\begin{array}{@{}#1@{}}#2\end{array}\right)%
    }{%
    \begin{pmatrix}#2\end{pmatrix}%
  }%
}

\newcommand{\dv}{\,|\,}

\NewDocumentCommand{\dx}{sO{x}}{\IfBooleanTF{#1}{}{\,}d#2}
\NewDocumentCommand{\ddx}{O{x}o}{\frac{d\IfValueTF{#2}{^{#2}}{}}{d#1\IfValueTF{#2}{^{#2}}{}}}
\NewDocumentCommand{\tddx}{O{x}}{\tfrac{d}{d#1}}

\NewDocumentCommand{\Mat}{mmo}{%
  [#1]_{%
  \IfValueTF{#3}{%
    #3\leftarrow#2%
    }{%
    #2}%
  }%
}

\NewDocumentCommand{\pd}{omm}{\frac{\partial\IfValueTF{#1}{^#1}{}#2}{\partial#3\IfValueTF{#1}{^#1}{}}}

\newcommand{\eeie}{\text{\ie}\quad}

\newcommand{\cfns}[2]{C(#1,#2)}

\NewDocumentCommand{\rzrep}{sm}{\IfBooleanTF{#1}{\left\langle#2\right\rangle}{\langle#2\rangle}}

\NewDocumentCommand{\defmap}{mmmmmo}{%
  \bea
  #1 : #2 &\ra& #3 \\*\relax
  #4 &\mapsto& #5 \IfValueTF{#6}{#6}{}
  \eea
}
\NewDocumentCommand{\givemap}{ommmmo}{%
  \bea
  \IfValueTF{#1}{#1 : }{}#2 &\ra& #3 \\
  #4 &\mapsto& #5 \IfValueTF{#6}{#6}{}
  \eea
}

\NewDocumentCommand{\vint}{O{\Zp}}{\int_{#1}}

\NewDocumentCommand{\ballb}{mmo}{%
  \overline{B}%
  \IfValueTF{#3}{_{#3}}{}%
  (#1,#2)%
}
\NewDocumentCommand{\ballNb}{mmo}{%
  B%
  \IfValueTF{#3}{_{#3}}{}%
  (#1,#2)%
}

\def\BibTeX{\textrm{B\kern-.05em\textsc{i\kern-.025em b}\kern-.08em T\kern-.1667em\lower.5ex\hbox{E}\kern-.125emX}} 
\NewDocumentCommand{\stirlf}{smm}{\operatorname{\mathnormal{s}}\IfBooleanTF{#1}{'}{}(#2,#3)}
\NewDocumentCommand{\stirls}{smm}{\operatorname{\mathnormal{S}}\IfBooleanTF{#1}{'}{}(#2,#3)}
\NewDocumentCommand{\stirltf}{sm}{\operatorname{St}\IfBooleanTF{#1}{_1}{}(#2)}
\NewDocumentCommand{\stirltfinv}{sm}{\operatorname{St}^{-1}\IfBooleanTF{#1}{_1}{}(#2)}
\NewDocumentCommand{\stirltfp}{sm}{\operatorname{St}'\IfBooleanTF{#1}{_1}{}(#2)}
\NewDocumentCommand{\stirltfinvp}{sm}{(\operatorname{St}')^{-1}\IfBooleanTF{#1}{_1}{}(#2)}

\NewDocumentCommand{\lt}{sm}{%
  \calL
  \IfBooleanTF{#1}{%
    \left\{
  }{\{}%
  #2
  \IfBooleanTF{#1}{%
    \right\}
  }{\}}%
}

\NewDocumentCommand{\vp}{O{p}}{v_{#1}}

\NewDocumentCommand{\contf}{smoo}{%
  \IfBooleanTF{#1}{(}{[}%
  #2%
  \IfValueTF{#3}{;#3}{}%
  \IfValueTF{#4}{%
    \ifx#3~
    \else
      ,%
    \fi
    \overline{#4}%
  }{}%
  \IfBooleanTF{#1}{)}{]}%
}

\NewDocumentCommand{\seqop}{smo}{%
  \operatorname{\mathbf{#2}}%
  \IfValueTF{#3}{%
    \mathopen{}%
    \IfBooleanTF{#1}{\left(}{(}%
    #3%
    \IfBooleanTF{#1}{\right)}{)}%
  }{}%
}
\NewDocumentCommand{\tfent}{smmm}{
  \arraycolsep=0.5pt%
  \IfBooleanTF{#1}{\left\{}{\left[}%
  \begin{array}{c} #3 \\ #4 \end{array}%
  \IfBooleanTF{#1}{\right\}}{\right]}_{#2}
}

\begin{document}

\maketitle

\begin{abstract}
Let $r$ be a non-zero rational number. In \cite{o'desky-richman:incomplete-gamma}, a construction was given of a $p$-adic incomplete gamma-function $\Gamma_p(\cdot,r)$ for each prime $p$ for which $|r - 1|_p < 1$. Aside from the special case where $r = 1$, only finitely many primes satisfy that condition for a given $r$, so it is desirable to lessen this restriction. In the present paper, we give a construction that works under the much weaker condition that $|r|_p = 1$ using a $p$-adic integral transform we introduced in \cite{buckingham:factorial}, which we interpret here as a $p$-adic analogue of an incomplete Mellin transform. For any given $r$, the condition $|r|_p = 1$ holds for all \emph{except} finitely many primes $p$.

Our approach emphasizes the parallels between the complex and $p$-adic constructions, explaining how a $p$-adic integration-by-parts formula takes the place of complex integration by parts in the proof of the recurrence relations for the $p$-adic incomplete gamma-functions. We introduce a two-variable $p$-adic transform for the task, extending our earlier $p$-adic integral transform.
\end{abstract}

\section{Introduction}

The well-known $\Gamma$-function, when considered first as a function $\Gamma : \R_{> 0} \ra \R$, may be defined by
\[ \Gamma(s) = \int_0^\infty t^{s - 1}\exp(-t)\dx[t] .\]
Because it plays an important role in number theory, it is natural to seek $p$-adic analogues of it, and indeed one exists, Morita's $p$-adic $\Gamma$-function $\Gamma_p : \Zp \ra \Cp$, actually taking values in the unit group $\Zp\st$ of $\Zp$. It holds its position as a $p$-adic analogue of $\Gamma$ not because of any interpolation property, as in the case of $p$-adic $L$-functions, but instead by virtue of several striking similarities, including, but not limited to, the following: (i) it has a definition that gives it the semblence of a factorial, (ii) it satisfies functional equations similar to the equations $\Gamma(s + 1) = s\Gamma(s)$ and $\Gamma(1 - s)\Gamma(s) = \pi/\sin(\pi s)$, and (iii) it enjoys a version of the Gauss multiplication formula.

It was shown in \cite{o'desky-richman:incomplete-gamma} that close cousins of the $\Gamma$-function, the incomplete gamma-functions, in fact can be $p$-adically interpolated in a very direct manner. If $r$ is chosen, for simplicity in this introduction, to be a positive real number, then $\Gamma(\cdot,r)$, when viewed as a function $\R \ra \R$, satisfies
\[ \Gamma(s,r) = \int_r^\infty t^{s - 1}\exp(-t)\dx[t] \]
for all $s \in \R$. If $r \in \Z_{\geq 1}$ and $p$ is a prime such that $r \equiv 1 \Mod p$, then the interpolation property shown in \cite{o'desky-richman:incomplete-gamma} was of the form $\tau_p(\Gamma(m,r)) = \Gamma_p(m,r)$ for all $m \in \Z_{\geq 1}$, where $\tau_p : \Q(e) \ra \Qp$ is a field embedding mapping the real number $e$ to a particular choice of transcendental element in $\Qp$. O'Desky and Richman established their interpolation property by using a theorem they proved, in the same paper, that gives a novel characterization of functions $\Z_{\geq 1} \ra \Qp$ in a certain class that can be extended by continuity to $\Zp \ra \Qp$.

In \cite{buckingham:factorial}, the present author subsequently obtained the same interpolation property, without the restriction that $r$ be an integer, but again under the assumption that $r \equiv 1 \Mod p$ (in the general sense that $|r - 1|_p < 1$). The methods in our paper were unrelated to those of \cite{o'desky-richman:incomplete-gamma}, using instead an isometric automorphism, which we called $\calS$, of the $\Cp$-Banach space $\cfns{\Zp}{\Cp}$ of continuous functions $\Zp \ra \Cp$, an automorphism that restricts to one of the space $\lan{\Zp}{\Cp}$ of locally analytic functions $\Zp \ra \Cp$. We also introduced a $p$-adic integral transform $\inttmap$ and proved, again in the case where $|r - 1|_p < 1$, an integral-transform formula for O'Desky and Richman's $p$-adic incomplete gamma-functions, akin to the formula for the classical incomplete gamma-functions.

The goal of the present paper is two-fold: (i) to define, for any given non-zero rational number $r$, a $p$-adic incomplete gamma-function not only when $p$ is one of the finitely many primes for which $|r - 1|_p < 1$, but under the weaker assumption that $|r|_p = 1$, an assumption that holds for all except finitely many $p$, and (ii) to show that our function interpolates $\Gamma(\cdot,r)$ on the set $1 + (p - 1)\Z_{\geq 0}$, which is dense in $\Zp$. This brings the theory of $p$-adic incomplete gamma-functions more into line with that of $p$-adic $L$-functions, where likewise the set of primes for which interpolation is possible is not finite, and where the set of integers on which the interpolation takes place is also defined by a congruence condition mod $p - 1$ (mod $2$ if $p = 2$); \cite[p.~57]{washington:cyc}.

At the heart of our interpolation strategy is a power series $f_r(t) \in \Q\pwr{t}$ that is common to both the classical and the $p$-adic incomplete gamma-functions\dsh common in the sense that the classical and the $p$-adic functions alike can be constructed directly from $f_r(t)$ via processes that, while different, are nonetheless parallel. Defined for any non-zero rational number $r$ by
\[ f_r(t) = -r((1 - t)^{\tfrac{1}{r}} - 1) = -r\sum_{k = 1}^\infty (-1)^k\ch{1/r}{k}t^k ,\]
it is a global object in the sense of having rational coefficients, but it gives rise to local objects, both at the archimedean place and at the non-archimedean ones. Broadly speaking, from $f_r(t)$ one constructs on the one hand a $\C$-valued function $\phi_{f_r} : \R_{\leq 0} \ra \C$ and on the other a $\Cp$-valued function $\phi_{f_r} : \Zp \ra \Cp$. In either case, we apply a function transform to $\phi_{f_r}$ that, very roughly speaking, produces an incomplete gamma-function. In the complex setting, the function transform is essentially a sort of incomplete Mellin transform, and in the $p$-adic setting, it is the transform $\inttmap$ that we introduced in \cite{buckingham:factorial}. The transformed functions, complex and $p$-adic, satisfy a functional equation that can be traced back to the equation
\[ f_r'(t) = (1 - t)^{\tfrac{1}{r} - 1} .\]
From the functional equation, we obtain a recurrence relation on $\Z_{\geq 0}$ satisfied by all transformed functions at once, the complex and the various $p$-adic functions, allowing us to prove our interpolation result.

The outline of the paper is as follows. After setting some key notation, we develop the $p$-adic theory in Sections~\ref{sec: p-adic integrals} to \ref{sec: 2-var}. Section~\ref{sec: p-adic integrals} reviews $\Cp$-valued measures on $\Zp$, Section~\ref{sec: pwr series corrs} sets out two power series correspondences we need in the $p$-adic theory, and in Section~\ref{sec: gexp} we describe how to produce a continuous function $\phi : \Zp \ra \Cp$ such that $\sum_{n = 0}^\infty \phi(n)\,\tfrac{t^n}{n!}$ is equal to $\exp(f(t))$ when $f(t)$ is a given power series, under some hypotheses on $f(t)$. Section~\ref{sec: 2-var} is the major $p$-adic section, in which we introduce a $2$-variable transform that unites the transforms $\calS$ and $\inttmap$ of \cite{buckingham:factorial}. The second of these, $\inttmap$, is a $p$-adic analogue of the complex incomplete Mellin transform. That section finishes with a proof of a $p$-adic analogue of classical integration by parts, which we note is different from the analogue introduced in \cite{cohen-friedman:raabe}. (Indeed, the two integration-by-parts results apply to two different types of $p$-adic integral.) In Section~\ref{sec: cmx inc gamma}, we review the construction of complex incomplete gamma-functions. Section~\ref{sec: transforms} gives a theorem on functional equations satisfied by functions obtained via the complex and $p$-adic incomplete Mellin transforms. Finally, in Section~\ref{sec: interpolation}, we prove our interpolation result for incomplete gamma-functions.

We end our introduction by remarking on the paper \cite{p-adic-incomp-and-hasse}, where another incomplete gamma-function was introduced (under the assumption $r \equiv 1 \Mod p$). That function is different from the one considered in \cite{o'desky-richman:incomplete-gamma} and \cite{buckingham:factorial}, being more akin to Morita's $\Gamma$-function. We will not consider the function of \cite{p-adic-incomp-and-hasse} in our paper.

The author would like to thank Al Weiss for always being willing to listen, Daniel Barsky for reading the manuscript and supplying interesting remarks, and Lana and Lynne for providing support in their own ways.

\section{Notation, definitions, and preliminary facts} \label{sec: notation}

Let $K$ be a field. We denote its group of non-zero elements by $K\st$. If $a(t) \in 1 + tK\pwr{t}$ and $y \in K$, we may set
\[ a(t)^y = \sum_{k = 0}^\infty \ch{y}{k}(a(t) - 1)^k ,\]
a well-defined element of $K\pwr{t}$. Formal manipulations of power series show that for all $a(t) \in 1 + tK\pwr{t}$ and all $y,z \in K$,
\begin{equation} \label{props of a(t)}
\left.
\begin{aligned}
(a(t)^y)' &= y\frac{a'(t)}{a(t)}a(t)^y , \quad \\
a(t)^{y + z} &= a(t)^y a(t)^z ,\\
a(t)^{yz} &= (a(t)^y)^z .
\end{aligned}
\right\}
\end{equation}

The following notions will be essential in the $p$-adic theory:
\begin{center}
\begin{longtable}{>{$}l<{$} @{:\quad} p{9cm}}
\cfns{\Zp}{\Cp} & The $\Cp$-vector space of continuous functions $\Zp \ra \Cp$ \\
\one & The function $x \mapsto 1$ \\
\xtox & The function $x \mapsto x$ \\
(x)_n & The element $x(x - 1)(x - 2) \cdots (x - (n - 1))$ of $\Zp$, where $n \in \Z_{\geq 0}$ and $x \in \Zp$ \\
\binf{n} & The function $x \mapsto \tfrac{1}{n!}(x)_n$, where $n \in \Z_{\geq 0}$ \\
\sigma & The forward-shift operator, satisfying $\sigma(\phi)(x) = \phi(x + 1)$ if $\phi \in \cfns{\Zp}{\Cp}$ \\
\nabla & The forward-difference operator, $\nabla = \sigma - \id$, where $\id$ is the identity operator; thus, $(\nabla\phi)(x) = \phi(x + 1) - \phi(x)$ \\
\pint & The subring $\{x \in \Cp \sat |x| \leq 1\}$ of $\Cp$ \\
\maxid & The ideal $\{x \in \Cp \sat |x| < 1\}$ of $\pint$ \\
\units & The subgroup $\{x \in \Cp\st \sat |x| = 1\}$ of $\Cp\st$ \\
\punits & The subgroup $\{x \in \Cp\st \sat |x - 1| < 1\} \con \units$ consisting of principal units in $\Cp$ \\
\|\cdot\| & The sup-norm on a function space, usually the space $\cfns{\Zp}{\Cp}$
\end{longtable}
\end{center}

\section{Integrals with respect to $p$-adic measures} \label{sec: p-adic integrals}

By a \emph{measure}, we will mean a bounded linear functional $\mu : \cfns{\Zp}{\Cp} \ra \Cp$, bounded in the sense that there is $C > 0$ such that $|\mu(\phi)| \leq C\|\phi\|$ for all $\phi \in \cfns{\Zp}{\Cp}$. If $M(\Zp,\Cp)$ denotes the set of measures and $\Cp\pwr{t}_{\mathrm{b}}$ denotes the set of power series in $\Cp\pwr{t}$ with bounded coefficients, then of course both are $\Cp$-vector spaces, and there are mutually inverse $\Cp$-linear maps
\bea
M(\Zp,\Cp) &\ra& \Cp\pwr{t}_{\mathrm{b}} \\*
\mu &\mapsto& \sum_{n = 0}^\infty \mu\left(\binf{n}\right)t^n \\[2ex]
\Cp\pwr{t}_{\mathrm{b}} &\ra& M(\Zp,\Cp) \\*
\sum_{n = 0}^\infty b_n t^n &\mapsto& \left(\sum_{n = 0}^\infty a_n\binf{n} \mapsto \sum_{n = 0}^\infty a_n b_n\right) .
\eea
If $\mu \in M(\Zp,\Cp)$ and $\phi \in \cfns{\Zp}{\Cp}$, then $\mu(\phi)$ will be denoted
\[ \int_{\Zp} \phi\dx[\mu] .\]

Two measures relevant for this paper are what we will call $\dirac[x]$, where $x \in \Zp$, and $\mu_{\psi,x}$, where $\psi \in \cfns{\Zp}{\Cp}$, given by the correspondences
\begin{align*}
\dirac[x] &\leftrightarrow (1 + t)^x ,\\
\mu_{\psi,x} &\leftrightarrow \sum_{n = 0}^\infty \ch{x}{n}\psi(x - n)t^n .
\end{align*}
If $\phi = \sum_{n = 0}^\infty a_n\binf{n}$, then because $(1 + t)^x = \sum_{n = 0}^\infty \ch{x}{n}t^n$, $\dirac[x]$ satisfies
\[ \int_{\Zp} \phi\dx[\dirac[x]] = \sum_{n = 0}^\infty a_n\ch{x}{n} = \phi(x) ,\]
that is, $\dirac[x]$ is the Dirac measure at $x$. The case $x = -1$ will be used frequently and will be denoted simply $\dirac$. Thus,
\[ \int_{\Zp} \phi\dx[\dirac] = \phi(-1) .\]
Of course, $\dirac[x] = \mu_{\one,x}$.

\section{Power series correspondences} \label{sec: pwr series corrs}

We recall the convolution $\myst$ introduced in \cite[(5.1)--(5.2)]{buckingham:factorial}. If $\decayegf$ denotes the set of power series $\sum_{n = 0}^\infty a_n\tfrac{t^n}{n!} \in \Cp\pwr{t}$ such that $a_n \to 0$ as $n \to \infty$, then the map
\bea
\prodcorr : \cfns{\Zp}{\Cp} &\ra& \decayegf \\*
\phi &\mapsto& \sum_{n = 0}^\infty (\nabla^n\phi)(0)\,\frac{t^n}{n!}
\eea
is a $\Cp$-vector space isomorphism, and the convolution $\myst$ on $\cfns{\Zp}{\Cp}$ is defined so that if $\cfns{\Zp}{\Cp}$ has the product $\myst$ instead of the usual pointwise product, then $\prodcorr$ is in fact an isomorphism of $\Cp$-algebras:
\begin{align}
\prodcorr(\phi \myst \psi) &= \prodcorr(\phi)\prodcorr(\psi) , \label{conv prod rule} \\
\eeie \sum_{n = 0}^\infty (\nabla^n(\phi \myst \psi))(0)\,\frac{t^n}{n!} &= \left(\sum_{n = 0}^\infty (\nabla^n\phi)(0)\,\frac{t^n}{n!}\right)\left(\sum_{n = 0}^\infty (\nabla^n\psi)(0)\,\frac{t^n}{n!}\right) . \nonumber
\end{align}

\begin{prop} \label{myst and norm}
If $\phi,\psi \in \cfns{\Zp}{\Cp}$, then $\|\phi \myst \psi\| \leq \|\phi\|\,\|\psi\|$.
\end{prop}

\begin{proof}
We use the fact that the sup-norm of a function in $\cfns{\Zp}{\Cp}$ is equal to the supremum of the absolute values of its Mahler coefficients \cite[Chp.~4, Sect.~2.4, Thm.~1]{robert:p-adic}. Let the Mahler series of $\phi$ and $\psi$ be
\[ \phi = \sum_{n = 0}^\infty a_n\binf{n} \quad\text{and}\quad \psi = \sum_{n = 0}^\infty b_n\binf{n} .\]
Then $\phi \myst \psi$ has Mahler series
\[ \phi \myst \psi = \sum_{n = 0}^\infty \left(\sum_{k = 0}^n \ch{n}{k}a_k b_{n - k}\right)\binf{n} ,\]
so
\[ \|\phi \myst \psi\| = \sup_{n \geq 0} \left|\sum_{k = 0}^n \ch{n}{k}a_k b_{n - k}\right| \leq \sup_{n \geq 0} \max_{k \leq n} |a_k|\,|b_{n - k}| \leq \|\phi\|\,\|\psi\| .\]
\end{proof}

In this paper, we will need the following other correspondence between continuous functions and power series. Let $\valueegf$ denote the set of power series $\sum_{n = 0}^\infty a_n\tfrac{t^n}{n!}$ such that the function $\Z_{\geq 0} \ra \Cp$ given by $n \mapsto a_n$ has an extension to a continuous function $\Zp \ra \Cp$. Then
\bea
\decayegf &\ra& \valueegf \\*
G(t) &\mapsto& \exp(t)G(t)
\eea
is an isomorphism of $\Cp$-vector spaces, and if $\phi \in \cfns{\Zp}{\Cp}$, then this isomorphism has the effect
\[ \sum_{n = 0}^\infty (\nabla^n\phi)(0)\,\frac{t^n}{n!} \mapsto \sum_{n = 0}^\infty \phi(n)\,\frac{t^n}{n!} ;\]
see \cite[Chap.~4, Sect.~1.1, Comment (2)]{robert:p-adic}. We consider, then, the isomorphism of $\Cp$-vector spaces
\bea
\actcorr : \cfns{\Zp}{\Cp} &\ra& \valueegf \\*
\phi &\mapsto& \exp(t)\prodcorr(\phi) = \sum_{n = 0}^\infty \phi(n)\,\frac{t^n}{n!} .
\eea
This map no longer respects multiplication but instead respects $\cfns{\Zp}{\Cp}$-actions, once those actions are defined correctly. Specifically, if we define actions on $\cfns{\Zp}{\Cp}$ and $\valueegf$ by, respectively,
\begin{align*}
\phi \cdot \psi &= \phi \myst \psi ,\\
\phi \cdot G(t) &= \prodcorr(\phi)G(t) ,
\end{align*}
then $\actcorr$ is a $\cfns{\Zp}{\Cp}$-module isomorphism. Indeed, multiplying both sides of the equality
\[ \prodcorr(\phi \myst \psi) = \prodcorr(\phi)\prodcorr(\psi) \]
by $\exp(t)$ gives
\beq \label{conv act rule}
\actcorr(\phi \myst \psi) = \prodcorr(\phi)\actcorr(\psi) .
\eeq
Further,
\beq \label{shiftu and diff}
\actcorr(\sigma(\phi)) = \actcorr(\phi)' ,
\eeq
simply because
\[ \sum_{n = 0}^\infty \sigma(\phi)(n)\,\frac{t^n}{n!} = \sum_{n = 0}^\infty \phi(n + 1)\,\frac{t^n}{n!} = \ddx[t]\left(\sum_{n = 0}^\infty \phi(n)\,\frac{t^n}{n!}\right) .\]
As we will see later, the correspondence given by $\actcorr$ will be invaluable for the theory of $p$-adic incomplete Mellin transforms.

\section{Continuous functions via $\gexp$} \label{sec: gexp}

Let $p$ be a prime. Define
\[ \expdom = \left\{x \in \Cp \SAT |x| < \oop^{\tfrac{1}{p - 1}}\right\} ,\]
on which the exponential series $\exp$ converges. If $f(t) \in \Cp\pwr{t}$ and $f(0) \in \expdom$, define $\gexp(f(t)) \in \Cp\pwr{t}$ by
\[ \gexp(f(t)) = \exp(f(0))\exp(f(t) - f(0)) .\]
(Note that $\exp(f(t) - f(0))$ converges in the formal power series ring $\Cp\pwr{t}$ because $f(t) - f(0)$ has positive $t$-adic valuation.) One verifies easily that the usual formulas
\[ \gexp(f(t) + g(t)) = \gexp(f(t))\gexp(g(t)) \quad\text{and}\quad \ddx[t](\gexp(f(t))) = f'(t)\gexp(f(t)) \]
hold.

\begin{lemma} \label{digit sum lemma}
Let $\ell \geq 2$ be an integer. If $n$ is a positive integer, $S_\ell(n)$ is its digit sum in base $\ell$, and $\log_\ell(n)$ is its logarithm to the base $\ell$, then
\[ S_\ell(n) \leq (\ell - 1)(\log_\ell(n) + 1) .\]
\end{lemma}

\begin{proof}
Write $n = \sum_{i = 0}^k a_i\ell^i$ where $k \in \Z_{\geq 0}$, $a_i \in \{0,\ldots,\ell - 1\}$ for each $i \in \{0,\ldots,k\}$, and $a_k \geq 1$. Then on the one hand, $n \geq \ell^k$, so $\log_\ell(n) \geq k$, but on the other hand, the fact that $a_i \leq \ell - 1$ for all $i$ gives
\[ S_\ell(n) \leq (\ell - 1)(k + 1) \leq (\ell - 1)(\log_\ell(n) + 1) .\]
\end{proof}

\begin{prop} \label{convergence of egf coeffs}
Let $f(t) \in \pint\pwr{t}$ be a power series such that $f(0) \in \expdom$ and $f'(0) \in \punits$. Then there is a unique $\phi \in \cfns{\Zp}{\Cp}$ such that
\[ \gexp(f(t)) = \sum_{n = 0}^\infty \phi(n)\,\frac{t^n}{n!} = \actcorr(\phi) \]
(as an equality of formal power series).
\end{prop}

\begin{proof}
We will make use of a result on the radius of convergence of a composition of $p$-adic power series, relying on the notion of the \emph{growth modulus}. We refer the reader to \cite[Chap.~6, Sect.~1.4]{robert:p-adic} for a discussion of the growth modulus and the related notion of critical radius.

Let $g(t) = f(t) - f(0) - t$, observe that $g(t) \in t\pint\pwr{t}$ and $g'(0) \in \maxid$, and write
\[ \exp(g(t)) = \sum_{n = 0}^\infty d_n\,\frac{t^n}{n!} \]
where $d_n \in \Cp$ for all $n$. We show that $d_n \to 0$ as $n \to \infty$.

Let $b = g'(0)$, and write $|b| = \oop^\alpha$ where $\alpha \in \Q_{> 0}$. Noting that $g(t) \in \pint\pwr{t}$ has radius of convergence at least $1$, choose $\beta \in \Q$ such that $0 < \beta < \min(\alpha,\tfrac{1}{2(p - 1)})$ and $\oop^{\tfrac{1}{p - 1} - \beta}$ is not a critical radius of $g(t)$. Then choose $\tau \in \Cp$ such that
\[ |\tau| = \oop^{\tfrac{1}{p - 1} - \beta} ,\]
and observe that
\begin{itemize}
\item $|\tau| < 1$, so $|\tau^2| \geq |\tau^k|$ for all $k \geq 2$,
\item $|\tau| = \oop^{\tfrac{1}{p - 1} - \beta} > \oop^{\tfrac{1}{p - 1}}$ because $\beta > 0$,
\item $|b\tau| = \oop^{\alpha + \tfrac{1}{p - 1} - \beta} < \oop^{\tfrac{1}{p - 1}}$ because $\beta < \alpha$,
\item $|\tau^2| = \oop^{2\big(\tfrac{1}{p - 1} - \beta\big)} < \oop^{\tfrac{1}{p - 1}}$ because $\beta < \tfrac{1}{2(p - 1)}$ so $2\big(\tfrac{1}{p - 1} - \beta\big) > \tfrac{1}{p - 1}$.
\end{itemize}
Now, if $s \mapsto M_s(g)$ is the growth modulus of $g$, then because $|\tau|$ is not a critical radius of $g$, we have $M_{|\tau|}(g) = |g(\tau)|$, so
\[ M_{|\tau|}(g) = |g(\tau)| \leq \max(|b\tau|,|\tau^2|) < \oop^{\tfrac{1}{p - 1}} ,\]
and so the radius of convergence of the power series $G(t) = \exp(g(t))$ is greater than $|\tau|$ by the theorem in \cite[Chp.~VI, Section~1.5]{robert:p-adic}. Therefore, $G(\tau)$ converges, so because
\[ G(t) = \exp(g(t)) = \sum_{n = 0}^\infty d_n\,\frac{t^n}{n!} ,\]
$d_n\frac{\tau^n}{n!} \to 0$. Write
\[ |d_n| = \oop^{v_n} \quad\text{and}\quad \left|d_n\frac{\tau^n}{n!}\right| = \oop^{u_n} \eqcomb{$v_n,u_n \in \Q$}[.] \]
The convergence of $G(\tau)$ says that $u_n \to \infty$. Now, using a well-known formula for $\vp(n!)$ (e.g., \cite[Chap.~5, Sect.~3.1]{robert:p-adic}), we have
\[ u_n = v_n + n\Big(\tfrac{1}{p - 1} - \beta\Big) - \frac{n - S_p(n)}{p - 1} = v_n - n\beta + \frac{S_p(n)}{p - 1} ,\]
so
\[ v_n = u_n + n\beta - \frac{S_p(n)}{p - 1} \geq u_n + n\beta - \log_p(n) - 1 \]
by Lemma~\ref{digit sum lemma}, so $v_n \to \infty$, and so $d_n \to 0$.

Now let $\phi = \exp(f(0))\sum_{n = 0}^\infty d_n\binf{n}$, noting that $\exp(f(0))$ is well defined because $f(0) \in \expdom$. Then
\begin{align*}
\actcorr(\phi) = \exp(t)\prodcorr(\phi) = \exp(t)\exp(f(0))\sum_{n = 0}^\infty d_n\,\frac{t^n}{n!} &= \exp(t)\exp(f(0))\exp(g(t)) \\
&= \exp(f(0))\exp(f(t) - f(0)) \\
&= \gexp(f(t)) .
\end{align*}
The uniqueness of $\phi$ is immediate from denseness and continuity.
\end{proof}

\section{A $2$-variable $p$-adic transform} \label{sec: 2-var}

Fix a prime $p$. We introduce a $2$-variable $p$-adic transform that unites the two transforms $\calS$ and $\inttmap$ introduced in \cite{buckingham:factorial}. It also allows us to generalize the notation $\calS^m$ found in \cite{buckingham:factorial}, where $m \in \Z$, to $\calS^y$ for an arbitrary $y \in \Zp$, a generalization that will be necessary later in the paper.

Let $\phi \in \cfns{\Zp}{\Cp}$. For $n \geq 0$, define
\defmap{\twovar{\phi}[n]}{\Zp \times \Zp}{\Cp}{(x,y)}{\sum_{k = 0}^n (-1)^k k!\,\ch{y}{k}\ch{x}{k}\phi(x - k)}[,]
which is continuous, each of the finitely many terms being given by a continuous function in $x$ multiplied by a polynomial expression in $x$ and $y$. If $m \geq n$, then
\begin{align}
|\twovar{\phi}[m](x,y) - \twovar{\phi}[n](x,y)| &= \left|\sum_{k = n + 1}^m (-1)^k k!\,\ch{y}{k}\ch{x}{k}\phi(x - k)\right| \nonumber \\
&\leq |(n + 1)!|\,\|\phi\| \label{twovar continuity prep}
\end{align}
for all $x,y \in \Zp$, so the theorem in \cite[Chap.~4, Sect.~2.1]{robert:p-adic} shows that we have a continuous function $\twovar{\phi} : \Zp \times \Zp \ra \Cp$ such that
\[ \twovar{\phi}(x,y) = \sum_{k = 0}^\infty (-1)^k k!\,\ch{y}{k}\ch{x}{k}\phi(x - k) \]
for all $x,y \in \Zp$.

Being continuous as a function on $\Zp \times \Zp$, $\twovar{\phi}$ is, in particular, continuous in each of its two variables individually. Continuity in the first variable says that for a fixed $y \in \Zp$, the function
\bea
\Zp &\ra& \Cp \\*
x &\mapsto& \twovar{\phi}(x,y)
\eea
is continuous, so we may define a $\Cp$-linear map
\begin{equation} \label{def of calS}
\begin{split}
\calS^y : \cfns{\Zp}{\Cp}\,\,\, &\ra \,\,\, \cfns{\Zp}{\Cp} \\*
\phi\,\,\, &\mapsto \,\,\, (x \mapsto \twovar{\phi}(x,y)) .
\end{split}
\end{equation}
We let $\calS = \calS^1$, so that
\[ \calS(\phi)(x) = \twovar{\phi}(x,1) = \phi(x) - x\phi(x - 1) .\]

Letting the second variable of $\twovar{\phi}$ vary instead, we have, for each $x \in \Zp$, a function
\bea
\Zp &\ra& \Cp \\*
y &\mapsto& \twovar{\phi}(x,y) .
\eea
This again is continuous, but examining its Mahler series,
\[ \sum_{k = 0}^\infty \left((-1)^k k!\,\ch{x}{k}\phi(x - k)\right)\binf{k} \]
(remember that $x$ is fixed here), we find further that it is locally analytic. This fact generalizes \cite[Prop.~6.4]{buckingham:factorial}, which was the special case $x = -1$, and the case of a general $x \in \Zp$ is proven in exactly the same way except that what is called $\psi(-1 - n)$ in the proof of \cite[Prop.~6.4]{buckingham:factorial} would be $(-1)^k\ch{x}{k}\phi(x - k)$ in the present context. The key result used is \cite[Cor.~I.4.8]{colmez:fonctions}; see also \cite[Sect.~10, Cor.~1]{amice:interpolation}.

Hence, denoting the $\Cp$-vector space of locally analytic functions $\Zp \ra \Cp$ by $\lan{\Zp}{\Cp}$, we have a $\Cp$-linear map
\defmap{\inttmap^x}{\cfns{\Zp}{\Cp}}{\lan{\Zp}{\Cp}}{\phi}{(y \mapsto \twovar{\phi}(x,y))}[.]
Immediately from the definitions,
\beq \label{inttmap-calS switch}
\inttmap^x(\phi)(y) = \twovar{\phi}(x,y) = \calS^y(\phi)(x)
\eeq
for all $\phi \in \cfns{\Zp}{\Cp}$ and all $x,y \in \Zp$. We set $\inttmap = \inttmap^{-1}$, so that
\[ \inttrans{\phi}(y) = \sum_{k = 0}^\infty k!\,\phi(-1 - k)\binf{k} ,\]
and obtain the special case
\beq \label{inttmap and calS relationship}
\inttrans{\phi}(y) = \calS^y(\phi)(-1)
\eeq
for all $\phi \in \cfns{\Zp}{\Cp}$ and all $y \in \Zp$.

\begin{prop} \label{continuity of y mapsto calS^y(phi) and x mapsto inttmap^x(phi)}
Let $\phi \in \cfns{\Zp}{\Cp}$. The maps
\bea
\Zp &\ra& \cfns{\Zp}{\Cp} \\*
y &\mapsto& \calS^y(\phi) \\[2ex]
\Zp &\ra& \lan{\Zp}{\Cp} \\*
x &\mapsto& \inttmap^x(\phi)
\eea
are continuous with respect to the sup-norm on each of $\cfns{\Zp}{\Cp}$ and $\lan{\Zp}{\Cp}$.
\end{prop}

\begin{proof}
We begin with the continuity of the first map. For a fixed $k \in \Z_{\geq 0}$, the map
\bea
\Zp &\ra& \cfns{\Zp}{\Cp} \\*
y &\mapsto& \left(x \mapsto (-1)^k k!\,\ch{y}{k}\ch{x}{k}\phi(x - k)\right)
\eea
is continuous. Indeed, if $y_1,y_2 \in \Zp$ and $x$ is also in $\Zp$,
\begin{align}
\Bigg|(-1)^k k!\,\ch{y_1}{k}\ch{x}{k}\phi(x - k) &- (-1)^k k!\,\ch{y_2}{k}\ch{x}{k}\phi(x - k)\Bigg| \nonumber \\*
&= |k!|\,\Bigg|\ch{x}{k}\Bigg|\,|\phi(x - k)|\,\Bigg|\ch{y_1}{k} - \ch{y_2}{k}\Bigg| \nonumber \\
&\leq \|\phi\|\,\Bigg|\ch{y_1}{k} - \ch{y_2}{k}\Bigg| , \label{ch{y_1}{k} - ch{y_2}{k}}
\end{align}
and the uniform continuity of $\binf{k}$ guarantees that for any $\eps > 0$, there is $\delta > 0$ such that the real number in (\ref{ch{y_1}{k} - ch{y_2}{k}}) is less than $\eps$ whenever $|y_1 - y_2| < \delta$. Hence, for any given $n \in \Z_{\geq 0}$, the map
\defmap{f_n}{\Zp}{\cfns{\Zp}{\Cp}}{y}{\Big(x \mapsto \twovar{\phi}[n](x,y)\Big)}
is a sum of $n + 1$ continuous maps as above ($k = 0,\ldots,n$) and is therefore continuous itself. Further, if $m,n \in \Z_{\geq 0}$ and $m \geq n$, then
\begin{align*}
\sup_{y \in \Zp} \|f_m(y) - f_n(y)\| &= \sup_{y \in \Zp} \left(\sup_{x \in \Zp} |f_m(y)(x) - f_n(y)(x)|\right) \\
&= \sup_{y \in \Zp} \left(\sup_{x \in \Zp} \Big|\twovar{\phi}[m](x,y) - \twovar{\phi}[n](x,y)\Big|\right) \\
&\leq |(n + 1)!|\,\|\phi\| \eqcom{by (\ref{twovar continuity prep})} ,
\end{align*}
and this tends to $0$ as $n \to \infty$. Therefore, applying again the theorem in \cite[Chap.~4, Sect.~2.1]{robert:p-adic}, this time with $\cfns{\Zp}{\Cp}$ as the codomain instead of $\Cp$, we see that the map obtained as the pointwise limit of the $f_n$ is continuous. But this map is precisely $y \mapsto (x \mapsto \twovar{\phi}(x,y))$, \ie $y \mapsto \calS^y(\phi)$.

Continuity of the second map is proven similarly by interchanging the roles of $x$ and $y$, with only these differences to be remarked upon: (i) The function $\binf{k}$ is replaced by $x \mapsto \ch{x}{k}\phi(x - k)$, which, of course, is again uniformly continuous. (ii) Because we have endowed $\lan{\Zp}{\Cp}$ with the sup-norm as well, continuity of the map $\Zp \ra \lan{\Zp}{\Cp}$ is equivalent to continuity of $\Zp \ra \lan{\Zp}{\Cp} \ra \cfns{\Zp}{\Cp}$, the second map here being inclusion, so in the application of the theorem in \cite[Chap.~4, Sect.~2.1]{robert:p-adic}, there is no harm in taking $\cfns{\Zp}{\Cp}$ rather than $\lan{\Zp}{\Cp}$ as the codomain.
\end{proof}

\subsection{Results on $\calS^y$}

\begin{prop} \label{(1 - t)^y}
If $\phi \in \cfns{\Zp}{\Cp}$ and $y \in \Zp$, then
\begin{align*}
(1 - t)^y\actcorr(\phi) &= \actcorr(\calS^y(\phi)) \\
\text{and}\quad (1 - t)^y\prodcorr(\phi) &= \prodcorr(\calS^y(\phi)) .
\end{align*}
\end{prop}

\begin{proof}
The two claimed equalities are equivalent via multiplication by $\exp(-t)$ and $\exp(t)$. We prove the first:
\begin{align*}
(1 - t)^y\sum_{n = 0}^\infty \phi(n)\,\frac{t^n}{n!} &= \left(\sum_{n = 0}^\infty (-1)^n\ch{y}{n}t^n\right)\left(\sum_{n = 0}^\infty \phi(n)\,\frac{t^n}{n!}\right) \\
&= \left(\sum_{n = 0}^\infty (-1)^n(y)_n\,\frac{t^n}{n!}\right)\left(\sum_{n = 0}^\infty \phi(n)\,\frac{t^n}{n!}\right) \\
&= \sum_{n = 0}^\infty \left(\sum_{k = 0}^n \ch{n}{k}(-1)^k(y)_k\phi(n - k)\right)\frac{t^n}{n!} \\
&= \sum_{n = 0}^\infty \twovar{\phi}(n,y)\,\frac{t^n}{n!} \\
&= \sum_{n = 0}^\infty \calS^y(\phi)(n)\,\frac{t^n}{n!} .
\end{align*}
\end{proof}

\begin{prop} \label{calS a Zp-action}
For all $y,z \in \Zp$, $\calS^{y + z} = \calS^y \of \calS^z$.
\end{prop}

\begin{proof}
We use Proposition~\ref{(1 - t)^y} in conjunction with the associativity of multiplication in the power series ring $\Cp\pwr{t}$. Specifically, if $\phi \in \cfns{\Zp}{\Cp}$,
\begin{align*}
\actcorr(\calS^{y + z}(\phi)) &= (1 - t)^{y + z}\actcorr(\phi) \\
&= (1 - t)^y(1 - t)^z\actcorr(\phi) \\
&= (1 - t)^y\actcorr(\calS^z(\phi)) \\
&= \actcorr(\calS^y(\calS^z(\phi))) .
\end{align*}
\end{proof}

\begin{cor} \label{isometric automorphism}
For each $y \in \Zp$, the map $\calS^y : \cfns{\Zp}{\Cp} \ra \cfns{\Zp}{\Cp}$ is an isometric automorphism of the $\Cp$-Banach space $\cfns{\Zp}{\Cp}$ with inverse $\calS^{-y}$.
\end{cor}

\begin{proof}
The invertibility follows immediately from the proposition upon observation of the fact that $\calS^0 = \operatorname{id}$, the identity map. As for isometry, the definition of $\calS^y(\phi)$ shows that $\|\calS^y(\phi)\| \leq \|\phi\|$ for all $y$ and all $\phi$, and then the relation $\calS^{-y} \of \calS^y = \operatorname{id}$ yields $\|\calS^y(\phi)\| = \|\phi\|$.
\end{proof}

The next corollary generalizes \cite[Theorem~6.8(ii)]{buckingham:factorial}, and in fact this generalization will be crucial later.

\begin{cor} \label{calS and shift for y in Zp}
If $\phi \in \cfns{\Zp}{\Cp}$ and $y,s \in \Zp$, then
\[ \inttrans{\calS^y(\phi)}(s) = \inttrans{\phi}(s + y) .\]
\end{cor}

\begin{proof}
If $\psi \in \cfns{\Zp}{\Cp}$ and $s \in \Zp$, then $\inttrans{\psi}(s) = \calS^s(\psi)(-1)$ by (\ref{inttmap and calS relationship}). Taking $\psi = \calS^y(\phi)$ gives
\begin{align*}
\inttrans{\calS^y(\phi)}(s) &= \calS^s(\calS^y(\phi))(-1) \\
&= \calS^{s + y}(\phi)(-1) \eqcom{by the proposition} \\
&= \inttrans{\phi}(s + y) .
\end{align*}
\end{proof}

\subsection{$\calS^y$ interpreted as convolution with $(\one - \xtox)^{\myst y}$}

Recall that $\one : \Zp \ra \Cp$ and $\xtox : \Zp \ra \Cp$ are the functions defined by
\begin{align*}
\one : x &\mapsto 1 ,\\
\xtox : x &\mapsto x .
\end{align*}
With this notation in hand, if $\phi \in \cfns{\Zp}{\Cp}$, then
\beq \label{calS and one - xtox}
\calS(\phi)(x) = \phi(x) - x\phi(x - 1) = \big((\one - \xtox) \myst \phi\big)(x)
\eeq
for all $x \in \Zp$. Later, we will give parallel expositions of complex and $p$-adic theories of integral transforms, a parallelism that will lead to $p$-adic interpolation of incomplete $\Gamma$-functions and other functions given by the incomplete Mellin transform. To emphasize this point of view, it will be benefical to build on the observation above that $\calS(\phi) = (\one - \xtox) \myst \phi$. What follows generalizes \cite[Prop.~5.5]{buckingham:factorial}.

The function $\one - \xtox$ corresponds, via $\prodcorr$, to the power series $1 - t$, which has a multiplicative inverse, so $\one - \xtox$ is invertible with respect to $\myst$. Its inverse is the function called $\qf$ in \cite{buckingham:factorial} and has Mahler series $\qf = \sum_{n = 0}^\infty n!\binf{n}$, and one also has $\qf = \calS^{-1}(\one)$. The function $\qf$ will appear in the proof of a $p$-adic integration by parts in Section~\ref{sec: int by parts}.

If $m \in \Z$, define $(\one - \xtox)^{(\myst m)}$ to be the $m$-fold convolution of $\one - \xtox$, as follows:
\[ (\one - \xtox)^{(\myst m)} =
\begin{cases}
\underbrace{(\one - \xtox) \myst \cdots \myst (\one - \xtox)}_m & \text{if $m \geq 0$,} \\
\underbrace{\qf \myst \cdots \myst \qf}_{-m} & \text{if $m < 0$.}
\end{cases}
\]

We now extend this definition. For a given $y \in \Zp$, define a function $(\one - \xtox)^{\myst y} \in \cfns{\Zp}{\Cp}$ by
\[ (\one - \xtox)^{\myst y} = \calS^y(\one) .\]
(We deliberately omit the parentheses in the exponent $\myst y$ in order to differentiate the notation from that for the $m$-fold convolution.)

\begin{prop} \label{extending m-fold conv of one - xtox}
The map
\bea
\Z &\ra& \cfns{\Zp}{\Cp} \\*
m &\mapsto& (\one - \xtox)^{(\myst m)}
\eea
has a unique extension to a continuous map $\Zp \ra \cfns{\Zp}{\Cp}$, and this extension is $y \mapsto (\one - \xtox)^{\myst y}$.
\end{prop}

\begin{proof}
The uniqueness is immediate via density. To show that $(\one - \xtox)^{\myst m} = (\one - \xtox)^{(\myst m)}$ for all $m \in \Z$, we can reduce easily to the case $m \geq 0$ and then proceed as follows:
\begin{align*}
\prodcorr((\one - \xtox)^{(\myst m)}) &= \prodcorr(\one - \xtox)^m \\
&= (1 - t)^m \\
&= (1 - t)^m\prodcorr(\one) \\
&= \prodcorr(\calS^m(\one)) \eqcom{by Proposition~\ref{(1 - t)^y}} \\
&= \prodcorr((\one - \xtox)^{\myst m}) \eqcom{by definition.}
\end{align*}
\end{proof}

\begin{prop} \label{calS acting via one - xtox}
For all $\phi \in \cfns{\Zp}{\Cp}$ and all $y \in \Zp$,
\[ \calS^y(\phi) = (\one - \xtox)^{\myst y} \myst \phi .\]
\end{prop}

\begin{proof}
The assertion is equivalent to showing that $\calS^y(\phi) = \calS^y(\one) \myst \phi$ for all $y \in \Zp$, and for this, we can use the bijection $\prodcorr$. Specifically,
\begin{align*}
\prodcorr(\calS^y(\one) \myst \phi) &= \prodcorr(\calS^y(\one))\prodcorr(\phi) \\
&= (1 - t)^y\prodcorr(\one)\prodcorr(\phi) \eqcom{by Proposition~\ref{(1 - t)^y}} \\
&= (1 - t)^y\prodcorr(\one \myst \phi) \\
&= (1 - t)^y\prodcorr(\phi) \\
&= \prodcorr(\calS^y(\phi)) \eqcom{by Proposition~\ref{(1 - t)^y} again.}
\end{align*}
\end{proof}

The following generalizes \cite[Theorem~6.12]{buckingham:factorial}, which applied in the special case where $x = -1$.

\begin{prop} \label{inttrans as an integral transform}
If $\phi \in \cfns{\Zp}{\Cp}$ and $x,y \in \Zp$, then
\begin{align*}
\inttmap^x(\phi)(y) &= \int_{\Zp} (\one - \xtox)^{\myst y}\dx[\mu_{\phi,x}] \\
&= \int_{\Zp} \big((\one - \xtox)^{\myst y} \myst \phi\big)\dx[\dirac[x]] .
\end{align*}
\end{prop}

\begin{proof}
If $y \in \Zp$,
\begin{align*}
\sum_{n = 0}^\infty \Big(\nabla^n((\one - \xtox)^{\myst y})\Big)(0)\,\frac{t^n}{n!} &= \prodcorr((\one - \xtox)^{\myst y}) \\
&= \prodcorr(\one - \xtox)^y \\
&= (1 - t)^y \\
&= \sum_{n = 0}^\infty (-1)^n\ch{y}{n}t^n \\
&= \sum_{n = 0}^\infty (-1)^n(y)_n\,\frac{t^n}{n!} ,
\end{align*}
so $\Big(\nabla^n((\one - \xtox)^{\myst y})\Big)(0) = (-1)^n(y)_n$. Therefore,
\begin{align*}
\inttmap^x(\phi)(y) &= \sum_{k = 0}^\infty (-1)^k(y)_k\ch{x}{k}\phi(x - k) \\
&= \sum_{k = 0}^\infty \Big(\nabla^k((\one - \xtox)^{\myst y})\Big)(0)\ch{x}{k}\phi(x - k) \\
&= \int_{\Zp} (\one - \xtox)^{\myst y}\dx[\mu_{\phi,x}] .
\end{align*}
For the equality with the other integral,
\begin{align*}
\inttmap^x(\phi)(y) &= \calS^y(\phi)(x) \\
&= \big((\one - \xtox)^{\myst y} \myst \phi\big)(x) \eqcom{by Proposition~\ref{calS acting via one - xtox}} \\
&= \int_{\Zp} \big((\one - \xtox)^{\myst y} \myst \phi\big)\dx[\dirac[x]] .
\end{align*}
\end{proof}

\subsection{A $p$-adic integration by parts} \label{sec: int by parts}

If $n \in \Z$, define $(\xtox - \one)^{\myst n} \in \cfns{\Zp}{\Cp}$ by
\[ (\xtox - \one)^{\myst n} = (-1)^n\calS^n(\one) .\]
As with the proof of Proposition~\ref{extending m-fold conv of one - xtox}, one may show that $(\xtox - \one)^{\myst n}$ is simply $n$-fold convolution of $\xtox - \one$ with itself (with the obvious definition if $n < 0$).

By Corollary~\ref{isometric automorphism}, $\|(\xtox - \one)^{\myst n}\| = \|\one\| = 1$ for all $n \in \Z$, so if $N \in \Z$ and $(a_n)_{n \leq N}$ is a family of elements of $\Cp$ such that $a_n \to 0$ as $n \to -\infty$, then we have a well-defined continuous function $\Zp \ra \Cp$ given by the infinite sum
\[ \sum_{n = -\infty}^N a_n(\xtox - \one)^{\myst n} .\]
We denote by $\amice$ the set of all functions of this form, a $\Cp$-vector space. We endow it with the sup-norm.

\begin{prop} \label{uniqueness of coeffs for amice functions}
Each $\phi \in \amice$ uniquely determines the elements $a_n$ for which $\phi = \sum_{n = -\infty}^N a_n(\xtox - \one)^{\myst n}$.
\end{prop}

\begin{proof}
It is enough to show that if $\sum_{n = -\infty}^N a_n(\xtox - \one)^{\myst n}$ is the zero function, then $a_n = 0$ for all $n$. For this, we note that for any $n \in \Z$,
\[ \calS^n(\one) = \calS^{n + 1}(\calS^{-1}(\one)) = \calS^{n + 1}(\qf) ,\]
where $\qf = \calS^{-1}(\one)$, so
\begin{align*}
\sum_{n = -\infty}^N a_n(\xtox - \one)^{\myst n} &= \sum_{n = -\infty}^N (-1)^n a_n\calS^n(\one) \\
&= \sum_{n = -\infty}^N (-1)^n a_n\calS^{n + 1}(\qf) \\
&= \calS^{N + 1}\left(\sum_{n = -\infty}^N (-1)^n a_n\calS^{n - N}(\qf)\right) \eqcom{by continuity} \\
&= \calS^{N + 1}\left(\sum_{n = -\infty}^0 (-1)^{n + N}a_{n + N}\calS^n(\qf)\right) \\
&= (-1)^N\calS^{N + 1}\left(\sum_{n = -\infty}^0 (-1)^n a_{n + N}\calS^n(\qf)\right) .
\end{align*}
The continuity step above uses Corollary~\ref{isometric automorphism}. Hence, if $\sum_{n = -\infty}^N a_n(\xtox - \one)^{\myst n}$ is the zero function, then so is
\[ (-1)^N\calS^{-(N + 1)}\left(\sum_{n = -\infty}^N a_n(\xtox - \one)^{\myst n}\right) = \sum_{n = -\infty}^0 (-1)^n a_{n + N}\calS^n(\qf) ,\]
and \cite[Prop.~6.18]{buckingham:factorial} tells us that all the $a_n$ are zero in this situation.
\end{proof}

In light of Proposition~\ref{uniqueness of coeffs for amice functions}, we may define a linear operator $\convderiv$ on $\amice$ by
\defmap{\convderiv}{\amice}{\amice}{\sum_{n = -\infty}^N a_n(\xtox - \one)^{\myst n}}{\sum_{n = -\infty}^N na_n(\xtox - \one)^{\myst(n - 1)}}[.]

\begin{prop} \label{integration by parts}
\phantom{}
\begin{enumerate}[label=(\roman*)]
\item $\calS^y \of \sigma = \sigma \of \calS^y + y\calS^{y - 1}$ for all $y \in \Zp$.\label{integration by parts: prep}
\item If $\phi \in \cfns{\Zp}{\Cp}$ and $x,y \in \Zp$, then
\begin{align*}
\inttmap^x(\sigma(\phi))(y) &= \inttmap^{x + 1}(\phi)(y) + y\inttmap^x(\phi)(y - 1) ,\\
\text{and}\quad \int_{\Zp} \big((\one - \xtox)^{\myst y} \myst (\sigma(\phi))\big)\dx[\dirac[x]] &= \int_{\Zp} \big((\one - \xtox)^{\myst y} \myst \phi\big)\dx[\dirac[x + 1]] \\*
\goinvisible &=\govisible {} + y\int_{\Zp} \big((\one - \xtox)^{\myst(y - 1)} \myst \phi\big)\dx[\dirac[x]] .
\end{align*}
\label{integration by parts: one term}
\item\textbf{(Integration by parts)} If $\psi \in \amice$, $\phi \in \cfns{\Zp}{\Cp}$, and $x \in \Zp$, then
\[ \int_{\Zp} \big(\psi \myst (\sigma(\phi))\big)\dx[\dirac[x]] = \int_{\Zp} (\psi \myst \phi)\dx[\dirac[x + 1]] - \int_{\Zp} \big(\convderiv(\psi) \myst \phi\big)\dx[\dirac[x]] .\]
\label{integration by parts: full}
\end{enumerate}
\end{prop}

\begin{proof}
\ref{integration by parts: prep} As is the case for classical integration by parts, the product rule plays a central role, but this time at the level of exponential generating functions. If $\phi \in \cfns{\Zp}{\Cp}$,
\begin{align*}
\actcorr(\sigma(\calS^y(\phi))) &= \actcorr(\calS^y(\phi))' \eqcom{by (\ref{shiftu and diff})} \\
&= \ddx[t]\Big((1 - t)^y\actcorr(\phi)\Big) \eqcom{by Proposition~\ref{(1 - t)^y}} \\
&= (1 - t)^y\actcorr(\phi)' - y(1 - t)^{y - 1}\actcorr(\phi) \\
&= (1 - t)^y\actcorr(\sigma(\phi)) - y(1 - t)^{y - 1}\actcorr(\phi) \eqcom{by (\ref{shiftu and diff})} \\
&= \actcorr(\calS^y(\sigma(\phi))) - y\actcorr(\calS^{y - 1}(\phi)) \eqcom{by Proposition~\ref{(1 - t)^y}} \\
&= \actcorr\Big(\calS^y(\sigma(\phi)) - y\calS^{y - 1}(\phi)\Big) ,
\end{align*}
so $\sigma(\calS^y(\phi)) = \calS^y(\sigma(\phi)) - y\calS^{y - 1}(\phi)$ by the injectivity of $\actcorr$.

\ref{integration by parts: one term}
\begin{align*}
\inttmap^x(\sigma(\phi))(y) &= \calS^y(\sigma(\phi))(x) \eqcom{by (\ref{inttmap-calS switch})} \\
&= \sigma(\calS^y(\phi))(x) + y\calS^{y - 1}(\phi)(x) \eqcom{by part~\ref{integration by parts: prep}} \\
&= \calS^y(\phi)(x + 1) + y\calS^{y - 1}(\phi)(x) \\
&= \inttmap^{x + 1}(\phi)(y) + y\inttmap^x(\phi)(y - 1) \eqcom{by (\ref{inttmap-calS switch}) again.}
\end{align*}
This proves the first equation, and then the second is just Proposition~\ref{inttrans as an integral transform}.

\ref{integration by parts: full} Fix $N \in \Z$, and let $X_N$ be the $\Cp$-Banach space, with sup-norm $|\cdot|$, consisting of sequences $\alpha = (a_n)_{n \leq N}$ of elements $a_n \in \Cp$ such that $a_n \to 0$ as $n \to -\infty$. If $\alpha = (a_n)_{n \leq N} \in X_N$, let
\[ \psi(\alpha) = \sum_{n = -\infty}^N a_n(\xtox - \one)^{\myst n} \in \amice .\]
Because $\|(\xtox - \one)^{\myst n}\| = 1$,
\[ \|\psi(\alpha)\| \leq \sup_{n \leq N} |a_n| = |\alpha| ,\]
so the map $X_N \ra \amice$ given by $\alpha \mapsto \psi(\alpha)$ is continuous. Similarly, if we define
\defmap{\tilde{\convderiv}}{X_N}{\amice}{(a_n)_{n \leq N}}{\sum_{n = -\infty}^N na_n(\xtox - \one)^{\myst(n - 1)}}[,]
then the observation
\[ \|\tilde{\convderiv}(\alpha)\| \leq \sup_{n \leq N} |na_n(\xtox - \one)^{\myst(n - 1)}| \leq \sup_{n \leq N} |a_n| = |\alpha| \]
shows that $\tilde{\convderiv}$, too, is continuous. But $\tilde{\convderiv}(\alpha) = \convderiv(\psi(\alpha))$. Hence, Proposition~\ref{myst and norm} and continuity of evaluation maps $\cfns{\Zp}{\Cp} \ra \Cp$ combine with the foregoing to show that the map $F_N : X_N \ra \Cp$ defined by
\[ F_N : \alpha \mapsto \int_{\Zp} \big(\psi(\alpha) \myst (\sigma(\phi))\big)\dx[\dirac[x]] - \int_{\Zp} (\psi(\alpha) \myst \phi)\dx[\dirac[x + 1]] + \int_{\Zp} \big(\convderiv(\psi(\alpha)) \myst \phi\big)\dx[\dirac[x]] \]
is continuous. Of course, $F_N$ is also $\Cp$-linear, so to show that it is the zero map, it is enough to show $F_N(\psi(\alpha)) = 0$ in the case where $\alpha$ has only one non-zero entry, corresponding to the situation where $\psi(\alpha) = (\xtox - \one)^{\myst n}$ for some $n \in \Z$. But the vanishing of $F_N(\alpha)$ in this case is immediate from part~\ref{integration by parts: one term} and the observation that $(\xtox - \one)^{\myst n} = (-1)^n(\one - \xtox)^{\myst n}$. We are now done, because for any $\psi \in \amice$, there are $N \in \Z$ and $\alpha \in X_N$ such that $\psi(\alpha) = \psi$.
\end{proof}

\section{Complex incomplete gamma-functions} \label{sec: cmx inc gamma}

Although complex incomplete gamma-functions are standard, a brief summary of their construction and relevant key properties will nonetheless be worthwhile. Everything in Section~\ref{inc gamma: defs} is well known, but we state the results and give short proofs partly for the benefit of the reader and partly to establish certain conventions held to in our paper.

\subsection{Definitions and basic properties} \label{inc gamma: defs}

If $\theta \in \R$, we let $\log_\theta$ denote the inverse of the bijection
\bea
\{z \in \C \sat \theta - 2\pi < \ig(z) \leq \theta\} &\ra& \C\st \\*
z &\mapsto& \exp(z) .
\eea
Thus, $\log_\theta$ is defined on all of $\C\st$, and further, it is analytic on $\C \setm R_\theta$ where $R_\theta$ is the ray
\[ R_\theta = \{a\exp(i\theta) \sat a \in \R_{\geq 0}\} .\]
If $w \in \C\st$ and $s \in \C$, one may define $w^s = \exp(s\log_\theta(w))$, omitting from the notation the dependence on $\theta$, as accords custom. The equality $w^{s_1 + s_2} = w^{s_1}w^{s_2}$ holds always, and this shows in particular that, if $m \in \Z$, then $w^m$ is simply the $m$-fold product of $w$ with itself and is consequently independent of $\theta$.

Often, we will refer to half-planes in $\C$, for which we introduce the following notation. If $a \in \R$, define
\begin{align*}
\C_{< a} &= \{s \in \C \sat \rl(s) < a\} ,\\
\C_{> a} &= \{s \in \C \sat \rl(s) > a\} .
\end{align*}

We make some definitions that will help us manage branches of the incomplete gamma-functions. Let $\calK$ be the set of maps $\kappa : \R_{\geq 0} \ra \C$ with the following properties:
\begin{enumerate}[label=(\roman*)]
\item For each $\nu \in \R_{\geq 0}$, the restriction $\kappa^{(\nu)} = \kappa|_{[0,\nu]} : [0,\nu] \ra \C$ is a contour.
\item $|\kappa(t)| \to \infty$ as $t \to \infty$.
\item There exist $\alpha_1,\alpha_2 \in (-\tfrac{\pi}{2},\tfrac{\pi}{2})$ and $t_0 \in \R_{\geq 0}$ such that for all $t \geq t_0$, $\kappa(t)$ has positive real part and lies between the rays $R_{\alpha_1}$ and $R_{\alpha_2}$.
\end{enumerate}

If $\theta \in \R$, we define a subset $\calK_\theta$ of $\calK$ by declaring that $\kappa \in \calK_\theta$ if its image in $\C$ does not intersect $R_\theta$ and, in the case where $\cos(\theta) > 0$, $\kappa$ satisfies the further condition that there exists $t_1 \in \R_{\geq 0}$ such that $\kappa(t)$ is on or below $R_\theta$ for all $t \geq t_1$. Then, if $r \in \C\st$, we let
\[ \calK_{\theta,r} = \{\kappa \in \calK_\theta \sat \kappa(0) = r\} .\]
We provide a sketch to illustrate the setup:
\begin{center}
\includegraphics{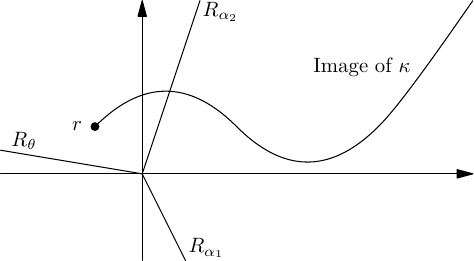}
\end{center}

In the following, we do not exclude the possibility that $\cos(\theta) > 0$, but we point out that for the purposes of incomplete gamma-functions, such a choice is not natural, because it results in a branch cut dividing the half-plane $\C_{> 0}$, the domain for the integration variable in the integrals defining the functions.

\begin{theorem}[Upper incomplete gamma-functions] \label{cmx gamma theorem}
Fix $\theta \in \R$. Let $r \in \C\st$, and choose any $\kappa \in \calK_{\theta,r}$. For each $N \in \Z_{\geq 0}$, define
\defmap{h_N}{\C}{\C}{s}{\int_{\kappa^{(N)}} u^{s - 1}\exp(-u)\dx[u]}[,]
with complex powers being defined via $\log_\theta$.
\begin{enumerate}[label=(\roman*)]
\item The functions $h_N$ converge pointwise as $N \to \infty$ to a function $\C \ra \C$, which we denote $\Gamma^{(\theta,\kappa)}(\cdot,r)$, and this function is analytic.\label{cmx gamma theorem: existence}
\item The function $\Gamma^{(\theta,\kappa)}(\cdot,r)$ in fact does not depend on the choice of $\kappa \in \calK_{\theta,r}$, and we denote it $\Gamma^\theta(\cdot,r)$.\label{cmx gamma theorem: indep of nu}
\end{enumerate}
\end{theorem}

\begin{proof}
Let $\kappa_1,\kappa_2 \in \calK_{\theta,r}$, and choose $\alpha_1 < \alpha_2$ in $(-\tfrac{\pi}{2},\tfrac{\pi}{2})$ such that the complex numbers $\kappa_1(t),\kappa_2(t)$ are, for large enough $t$, eventually always between the rays $R_{\alpha_1}$ and $R_{\alpha_2}$ in $\C_{> 0}$. If $\nu_1 \in \R_{\geq 0}$ is large enough, there is an origin-centred arc $A$ in $\C_{> 0}$, of some radius $\rho > 0$, that begins at $\kappa_1(\nu_1)$ and ends at some point $\kappa_2(\nu_2)$ without leaving the sector bounded by $R_{\alpha_1}$ and $R_{\alpha_2}$, and without crossing $R_\theta$ either. This is possible because of how $\calK_{\theta,r}$ is defined. By path independence, the difference
\[ \int_{\kappa_2^{(\nu_2)}} u^{s - 1}\exp(-u)\dx[u] - \int_{\kappa_2^{(\nu_2)}} u^{s - 1}\exp(-u)\dx[u] \]
equals an integral of the same integrand over the arc $A$. But there is $C_1 > 0$ such that
\[ \left|\int_A u^{s - 1}\exp(-u)\dx[u]\right| \leq C_1\rho^{\rl(s)}\int_{\alpha_1}^{\alpha_2} \exp(-\rho\cos(t))\dx[t] ,\]
and if we choose $C_2 > 0$ with $C_2 < \min(\cos(\alpha_1),\cos(\alpha_2))$, then this last integral is at most
\[ C_1\rho^{\rl(s)}\int_{\alpha_1}^{\alpha_2} \exp(-\rho C_2)\dx[t] = (\alpha_2 - \alpha_2)C_1\rho^{\rl(s)}\exp(-\rho C_2) .\]
For any $s \in \C$, $\rho^{\rl(s)}\exp(-\rho C_2) \to 0$ as $\rho \to 0$, so whether or not the integral $\int_\kappa u^{s - 1}\exp(-u)\dx[u]$ converges is independent of $\kappa$, and its value in the case of convergence is also independent of $\kappa$.

Let $\kappa \in \calK_{\theta,r}$ now be any convenient choice\dsh for simplicity, one that eventually coincides with some ray in $\C_{> 0}$. The integrand in the integral defining $h_N$ is continuous in $t$ for each $s$, and is analytic in $s$ for each $t$, so $h_N$ is analytic on $\C$; see \cite[Theorem~17.9]{bak-newman:complex-analysis}, for example. Now, fix a non-empty compact subset $S$ of $\C$. The fact that $t$ grows faster than $\log(t)$ as $t \in \R$ tends to $\infty$ shows that there are real numbers $B_S \in \R$, $C_S > 0$, and $N_S \geq 0$, depending only on $r$, $\theta$, $\kappa$, and $S$, such that for all integers $M,N$ with $N \geq M \geq N_S$,
\beq \label{h_N unif conv on compacts: key bound}
\Big\|h_N|_S - h_M|_S\Big\|_S \leq \int_M^N \exp(B_S - C_S t)\dx[t] ,
\eeq
where $\|\cdot\|_S$ is the sup-norm on continuous functions $S \ra \C$. The right-hand side of (\ref{h_N unif conv on compacts: key bound}) converges to zero as $M,N \to \infty$ with $N \geq M$, so the sequence of functions $(h_N|_S)_{N \geq 0}$ is Cauchy with respect to the sup-norm. We are now done by the fact that uniform convergence on compact subsets implies analyticity \cite[Theorem~7.6]{bak-newman:complex-analysis}.
\end{proof}

\begin{theorem}[Lower incomplete gamma-functions] \label{cmx lower gamma theorem}
Fix $\theta \in \R$. Let $r \in \C\st$, and for each $N \in \Z_{\geq 0}$, define
\defmap{h_N}{\C_{> 0}}{\C}{s}{\int_{1/N}^1 (tr)^{s - 1}\exp(-tr)r\dx[t]}[,]
with complex powers being defined via $\log_\theta$. Then the functions $h_N$ converge pointwise as $N \to \infty$ to a function $\C_{> 0} \ra \C$, which we denote $\gamma^\theta(\cdot,r)$, and this function is analytic.
\end{theorem}

\begin{proof}
The proof runs along similar lines to that of Theorem~\ref{cmx gamma theorem}, with the following differences: (i) For the analyticity of $h_N$, this time the crucial point is that the set $\{tr \sat t \in (0,1]\}$ is either disjoint from $R_\theta$ or contained entirely in it, both sets being radial. (ii) If $S$ is a non-empty compact subset of $\C_{> 0}$, there is $\eps > 0$ such that $\rl(s) \geq \eps$ for all $s \in S$, so if $t \in (0,1]$, $t^{\rl(s) - 1} \leq t^{\eps - 1}$. (iii) When bounding the integrand, the factor $|\mathopen{}\exp(-tr)|$ is bounded by a positive constant (\ie independent of $t$), and $|r^{s - 1}|$ is bounded by a constant independent of $s$ because $S$ is compact, so there is $C > 0$ such that
\[ |(tr)^{s - 1}\exp(-tr)r| \leq Ct^{\eps - 1} \]
for all $s \in S$ and all $t \in (0,1]$. Therefore, if $1 \leq M \leq N$,
\[ \sup_{s \in S} \left|\int_{1/N}^{1/M} (tr)^{s - 1}\exp(-tr)r\dx[t]\right| \leq C\int_{1/N}^{1/M} t^{\eps - 1}\dx[t] = \frac{1}{\eps}\left(\tfrac{1}{M^\eps} - \tfrac{1}{N^\eps}\right) ,\]
and this converges to zero as $M,N \to \infty$ with $N \geq M$.
\end{proof}

Define
\defmap{\gammahat}{\C_{> 0}}{\C}{s}{\Gamma^\theta(s,r) + \gamma^\theta(s,r)}[.]
It is irrelevant whether this function is equal to the usual $\Gamma$-function. All that matters for our purposes is that it has the same value at $1$ and the same functional equation, as demonstrated in Proposition~\ref{rec rels for both gammas}.

\begin{prop} \label{rec rels for both gammas}
Fix $\theta \in \R$, and let $r \in \C\st$. One has the following special values and functional equations:
\begin{align*}
\Gamma^\theta(1,r) &= \exp(-r) ,\\
\Gamma^\theta(s + 1,r) &= r^s\exp(-r) + s\Gamma^\theta(s,r) \eqcom{for all $s \in \C$,} \\
\gamma^\theta(1,r) &= 1 - \exp(-r) ,\\
\gamma^\theta(s + 1,r) &= -r^s\exp(-r) + s\gamma^\theta(s,r) \eqcom{for all $s \in \C_{> 0}$,} \\
\gammahat(1) &= 1 ,\\
\gammahat(s + 1) &= s\gammahat(s) \eqcom{for all $s \in \C_{> 0}$.}
\end{align*}
\end{prop}

\begin{proof}
The functional equations for $\Gamma^\theta(\cdot,r)$ and $\gamma^\theta(\cdot,r)$ are obtained via integration by parts, and the values at $1$ may be obtained via a primitive of $u \mapsto \exp(-u)$, or, in the case of $\Gamma^\theta(1,r)$, by taking $s = 0$ in the functional equation. The functional equation for $\gammahat$ is obtained by adding the functional equations of $\Gamma^\theta(s,r)$ and $\gamma^\theta(s,r)$, and $\gammahat(1)$ is similarly found.
\end{proof}

\brem \label{rem: independent of theta}
If $m \in \Z$, then $\Gamma^\theta(m,r)$ is independent of $\theta$, as is immediately apparent in light of the fact that $\exp(m\log_\theta(z))$ is independent of $\theta$ in that case. The same holds for $\gamma^\theta(m,r)$ when $m \in \Z_{\geq 1}$.
\erem

\subsection{A convenient form of the incomplete gamma-functions}

The following will assist us in relating $\Gamma^\theta(\cdot,r)$ and $\gamma^\theta(\cdot,r)$ to the integral transform we will consider later. If $\theta \in \R$, then for any $r \in \C\st$, define
\bea
\gfn^\theta(\cdot,r) : \C &\ra& \C \\*
s &\mapsto& \exp(r)\Gamma^\theta(s + 1,r) ,\\[2ex]
\lgfn^\theta(\cdot,r) : \C_{> -1} &\ra& \C \\*
s &\mapsto& \exp(r)\gamma^\theta(s + 1,r) .
\eea
Straight from the definitions, we have
\beq \label{Gamma as sum}
\gfn^\theta(s,r) + \lgfn^\theta(s,r) = \exp(r)\gammahat(s + 1) \eqcom{for all $s \in \C_{> -1}$.}
\eeq

\begin{prop} \label{standardized form of gfn}
Let $\theta \in \R$.
\begin{enumerate}[label=(\roman*)]
\item Suppose that $r \in \C_{> 0}$, and if $\cos(\theta) > 0$, assume that $r$ is on or below the ray $R_\theta$. Then for all $s \in \C$,
\[ \gfn^\theta(s,r) = r^{s + 1}\int_{-\infty}^0 (1 - x)^s\exp(rx)\dx ,\]
the integral being taken over the real half-line $(-\infty,0]$.\label{standardized form: gfn}
\item If $r \in \C_{< 0}$, then for all $s \in \C_{> -1}$,
\[ \lgfn^\theta(s,r) = r^{s + 1}\int_0^1 (1 - x)^s\exp(rx)\dx .\]
\label{standardized form: lgfn}
\end{enumerate}
\end{prop}

\begin{proof}
If $\rl(r) > 0$, then the assumptions ensure that the map $\kappa : \R_{\geq 0} \ra \C$ defined by $\kappa : t \mapsto r(1 + t)$ belongs to $\calK_{\theta,r}$. Thus, for all $s \in \C$,
\[ \gfn^\theta(s,r) = \exp(r)\int_0^\infty (r + tr)^s\exp(-(r + tr))r\dx[t] = r^{s + 1}\int_0^\infty (1 + t)^s\exp(-tr)\dx[t] .\]
The change of variables $x = -t$ completes the proof.

The proof for $\lgfn^\theta(s,r)$ for $r \in \C_{< 0}$ is similar. If $\rl(s) > -1$,
\[ \lgfn^\theta(s,r) = \exp(r)\int_0^1 (tr)^s\exp(-tr)r\dx[t] = r^{s + 1}\int_0^1 t^s\exp(r - tr)\dx[t] ,\]
and then we make the change of variables $x = 1 - t$.
\end{proof}

\section{The set $\trable$ and operations on functions in $\trable$} \label{sec: transforms}

\subsection{Complex analogues} \label{sec: complex analogues}

We aim to find operations in the complex setting that correspond to the $p$-adic operations $\calS$, $\sigma$, and $\inttmap$. The starting point is the formula
\beq \label{x = -1 parts}
\int_{\Zp} \Big((\one - \xtox)^{\myst y} \myst (\sigma(\phi))\Big)\dx[\dirac] = \phi(0) + y\int_{\Zp} ((\one - \xtox)^{\myst(y - 1)} \myst \phi)\dx[\dirac] ,
\eeq
which follows immediately from Proposition~\ref{integration by parts} in the case $x = -1$ via the observation that $\inttmap^0(\phi)(y) = \twovar{\phi}(0,y) = \phi(0)$. Proposition~\ref{integration by parts} and its special case (\ref{x = -1 parts}) are a form of integration by parts, with $\sigma$ playing a role of differentiation. Therefore, to arrive at a complex version of this theory, we look to choose $a,b \in \R \cup \{-\infty,\infty\}$, with $a < b$, such that the formula
\[ \int_a^b (1 - x)^y\phi'(x)\dx = \phi(0) + y\int_a^b (1 - x)^{y - 1}\phi(x)\dx \]
holds. It will do so as long as the term $[\phi(x)]_a^b$ in the usual integration-by-parts formula is equal to $\phi(0)$. The natural choices for $a$ and $b$, therefore, are $b = 0$ to obtain $\phi(0)$, and $a = -\infty$ with the assumption that $(1 - x)^y\phi(x) \to 0$ as $x \to -\infty$. This leads us to the space $\trable$ defined below. Note that for $x \in \R_{\leq 0}$ and $s \in \C$, $(1 - x)^s$ is always to be interpreted as $\exp(s\log(1 - x))$, where $\log$ is the usual real logarithm.

Let $\Omega$ be either $\C$ or $\Cp$. We will say that we are in the \emph{complex case} or the \emph{$p$-adic case} as applicable, and set
\begin{align*}
\pret &=
\begin{cases}
\R_{\leq 0} & \text{if $\Omega = \C$,} \\
\Zp & \text{if $\Omega = \Cp$,}
\end{cases} \\[2ex]
\postt &=
\begin{cases}
\C & \text{if $\Omega = \C$,} \\
\Zp & \text{if $\Omega = \Cp$.}
\end{cases}
\end{align*}
Elements in $\pret$ will often be denoted by $x$, and elements in $\postt$ by $s$ or sometimes $y$.

We will say that a function $f : \R_{\leq 0} \ra \C$ is \emph{differentiable} if the two-sided derivative exists at each point less than $0$ and the one-sided derivative (on the left) exists at $0$.

In the complex case, if $a \in \R$, define a set $\trable[a]$ of functions $\pret \ra \Omega$ by declaring that $\phi \in \trable[a]$ if it is differentiable and there exist $C > 0$ and $X \leq 0$ such that $|\phi(x)| \leq C(1 - x)^{-a - 1}$ for all $x \leq X$. We then let
\[ \trable = \bigcap_{a \in \R} \trable[a] .\]
In the $p$-adic case, we simply let $\trable = \cfns{\Zp}{\Cp}$ and do not consider the sets $\trable[a]$. Of course, all these sets are vector spaces over $\Omega$.

In the complex case, the condition for membership in $\trable[a]$ ensures that if $S$ is a non-empty compact subset of $\C_{< a}$ and $\eps$ is chosen such that $\rl(s) < a - \eps$ for all $s \in S$, then as long as $x \leq X$,
\[ |(1 - x)^s\phi(x)| \leq (1 - x)^{a - \eps}(1 - x)^{-a - 1} = (1 - x)^{-1 - \eps} .\]
Therefore, if $M,N \in \Z$, $M \leq N \leq X$, and $s \in S$,
\[ \left|\int_M^N (1 - x)^s\phi(x)\dx\right| \leq \int_M^N (1 - x)^{-1 - \eps}\dx = \frac{1}{\eps}((1 - N)^{-\eps} - (1 - M)^{-\eps}) ,\]
and this converges to zero as $M,N \to -\infty$ with $M \leq N$. Thus, we have a well-defined analytic function
\bea
\C_{< a} &\ra& \C \\*
s &\mapsto& \int_{-\infty}^0 (1 - x)^s\phi(x)\dx \eqcomb{$\phi \in \trable[a]$}[.]
\eea

The following is a useful sufficient condition for a function to be in $\trable$ in the complex case.

\begin{prop} \label{sufficiently general fn in trable}
Let $F : \R_{\leq 0} \ra \C$ be differentiable, and assume that there exist $X \in \R_{\leq -1}$ and $\nu \in \R_{> 0}$ such that $\rl(F(x)) \leq -(-x)^\nu$ for all $x \leq X$. If $g : \R_{\leq 0} \ra \C$ is a polynomial function or, more generally, a finite $\C$-linear combination of functions of the form $x \mapsto (1 - x)^y$ where $y \in \C$, then the function $\R_{\leq 0} \ra \C$, $x \mapsto g(x)\exp(F(x))$, is in $\trable$.
\end{prop}

\begin{proof}
If $\phi \in \trable$, then so is the map $x \mapsto (1 - x)^y\phi(x)$ for every $y \in \C$, so it is enough to prove the assertion in the case where the function $g$ is the constant function $x \mapsto 1$. For this, let $a \in \R$ and note that
\[ (a + 1)\log(1 - x) - (-x)^\nu \to -\infty \eqcom{as $x \to -\infty$,} \]
so
\[ \frac{\exp(-(-x)^\nu)}{(1 - x)^{-a - 1}} \to 0 \eqcom{as $x \to -\infty$.} \]
Therefore, there is $X_1 \leq 0$ such that $\exp(-(-x)^\nu) \leq (1 - x)^{-a - 1}$ for all $x \leq X_1$, so as long as $x \leq \min(X,X_1)$,
\[ |\mathopen{}\exp(F(x))| = \exp(\rl(F(x))) \leq \exp(-(-x)^\nu) \leq (1 - x)^{-a - 1} ,\]
so the function $x \mapsto \exp(F(x))$ is in $\trable[a]$.
\end{proof}

\subsection{Operations on functions $\phi \in \trable$}

As much as possible in the remainder, we will give an exposition that emphasizes the several parallels between the complex and $p$-adic sides of the theory. This will mean using the same symbols in both settings, such as $\Omega$ for both $\C$ and $\Cp$ as in Section~\ref{sec: complex analogues}, and stating the results with little or no reference to the particular case. Of course, the proofs of results will need to be broken up into the two cases, but even within proofs some parallels can be observed.

It will be convenient to use, in both the complex case and the $p$-adic, the notation $\one : \pret \ra \Omega$ and $\xtox : \pret \ra \Omega$ for the functions defined by
\begin{align*}
\one : x &\mapsto 1 ,\\
\xtox : x &\mapsto x .
\end{align*}

If $\phi \in \trable$ and $y \in \postt$, we have the function $\calS^y(\phi) : \pret \ra \Omega$ given by
\[ \calS^y(\phi) =
\begin{cases}
(\one - \xtox)^y\phi & \text{in the complex case,} \\
(\one - \xtox)^{\myst y} \myst \phi & \text{in the $p$-adic case.}
\end{cases}
\]
In the complex case, this is a definition, and in the $p$-adic case, we proved this fact in Proposition~\ref{calS acting via one - xtox}. We observe in passing the special case where $y = 1$:
\[ \calS^1(\phi)(x) =
\begin{cases}
\phi(x) - x\phi(x) & \text{in the complex case,} \\
\phi(x) - x\phi(x - 1) & \text{in the $p$-adic case.}
\end{cases}
\]

As we suggested in our discussion in Section~\ref{sec: complex analogues}, $\sigma$ should be the differentiation operator in the complex case. Thus, $\sigma(\phi) : \pret \ra \Omega$ is given by
\[ \sigma(\phi)(x) =
\begin{cases}
\phi'(x) & \text{in the complex case,} \\
\phi(x + 1) & \text{in the $p$-adic case (a definition in Section~\ref{sec: notation}),}
\end{cases}
\]
and then it is natural to define $\gent{\phi} : \postt \ra \Omega$ in the complex case so as to match its description as an integral transform in the $p$-adic case:
\[ \gent{\phi}(s) =
\begin{cases}
\displaystyle \int_{-\infty}^0 (1 - x)^s\phi(x)\dx & \text{in the complex case,} \\[3ex]
\displaystyle \int_{\Zp} \big((\one - \xtox)^{\myst s} \myst \phi\big)\dx[\dirac] & \text{in the $p$-adic case.}
\end{cases}
\]
Again, this is a definition in the complex case and a fact in the $p$-adic. If $a \in \R$ and $\phi \in \trable[a]$ in the complex case, then the integral formula for $\gent{\phi}$ above defines a function $\C_{< a} \ra \C$.

We think of $\gent{\phi}$ as an incomplete Mellin transform or a $p$-adic incomplete Mellin transform of $\phi$, according to which case we are in. The reason for this viewpoint is that, in the complex case, the change of integration variable $t = 1 - x$ gives
\[ \gent{\phi}(s) = \int_1^\infty t^s\phi(1 - t)\dx[t] .\]
The usual Mellin transform, if it exists, of a function $\psi : \R_{> 0} \ra \C$ is given by
\[ M(\psi) : s \mapsto \int_0^\infty t^{s - 1}\psi(t)\dx[t] .\]
Thus, if we set $\psi(x) = \phi(1 - x)$, then $M(\psi)(s + 1)$ and $\gent{\phi}(s)$ are both integrals of $t^s\phi(1 - t)\dx[t]$, with the first being an integral from $0$ to $\infty$ and the second an integral from $1$ to $\infty$, thus being `incomplete'.

\subsection{The functions $\phi_f$ and $\Phi_f$} \label{sec: compatible}

We will say that a power series $f(t) \in \Omega\pwr{t}$ is \emph{compatible}, respectively, \emph{$a$-compatible} (where $a \in \R$), if
\begin{itemize}
\item (in the complex case) $f(t)$ has a positive radius of convergence $\rho_f$, and there exist a complex domain $V$ containing $\pret = \R_{\leq 0}$, a positive real number $c < \rho_f$, and an analytic function $\widehat{f} : V \ra \C$ such that
  \begin{enumerate}[label=(\roman*)]
  \item $\widehat{f}(x) = f(x)$ for all $x \in V$ with $|x| < c$, and
  \item the function $\pret \ra \C$, $x \mapsto \exp(\widehat{f}(x))$, is in $\trable$, respectively, $\trable[a]$.
  \end{enumerate}
\item (in the $p$-adic case) $f(t) \in \pint\pwr{t}$, $f(0) \in \expdom$, and $f'(0) \in \punits$.
\end{itemize}

If $f(t)$ is compatible, or $a$-compatible as needed in the complex case, we associate to it functions $\phi_f : \pret \ra \Omega$ and $\Phi_f : \postt \ra \Omega$.
\begin{itemize}
\item \textbf{Definition of $\phi_f$ in the complex case.} If $\widehat{f}_1 : V_1 \ra \C$ and $\widehat{f}_2 : V_2 \ra \C$ are two analytic functions as in the compatibility (or $a$-compatibility) definition above, then the identity theorem guarantees that $\widehat{f}_1$ and $\widehat{f}_2$ agree on $V_1 \cap V_2$ and therefore on $\pret$, so we may define $\phi_f : \pret \ra \C$ by
\[ \phi_f(x) = \exp(\widehat{f}(x)) \]
for any $\widehat{f}$, being independent of this choice. The function $\phi_f$ is in $\trable$ or $\trable[a]$, as applicable.
\item \textbf{Definition of $\phi_f$ in the $p$-adic case.} The compatibility condition ensures that $f(t)$ meets the hypotheses of Proposition~\ref{convergence of egf coeffs}, so we may define $\phi_f$ to be the unique function in $\cfns{\Zp}{\Cp} = \trable$ such that
\[ \actcorr(\phi_f) = \gexp(f(t)) .\]
\end{itemize}
Hence, in either case we define
\[ \Phi_f = \gent{\phi_f} ,\]
a function $\postt \ra \Omega$ if $f(t)$ is compatible, and a function $\C_{< a} \ra \Omega$ if we are in the complex case and $f(t)$ is $a$-compatible.

\subsection{A functional equation for $\Phi_f$ in the polynomial case}

\begin{lemma} \label{gent and sigma}
Let $\phi \in \trable$. If $\sigma(\phi) \in \trable$ (automatically true in the $p$-adic case), then
\[ \gent{\sigma(\phi)}(s) = \phi(0) + s\gent{\phi}(s - 1) \eqcom{for all $s \in \postt$.} \]
In the complex case, if $\phi \in \trable[a]$ and $\phi' \in \trable[a']$, then the same equation holds as long as $\rl(s) < \min(a + 1,a')$.
\end{lemma}

\begin{proof}
The $p$-adic case is obtained by taking $x = -1$ in Proposition~\ref{integration by parts}, which we recall is a $p$-adic analogue of integration by parts.

For the complex case, it is enough to prove the second assertion, which is also done by integration by parts. If $\phi \in \trable[a]$, then there are $C > 0$ and $X \leq 0$ such that $|\phi(x)| \leq C(1 - x)^{-a - 1}$ for all $x \leq X$. Therefore, if $\rl(s) < a + 1$, say $\rl(s) = a + 1 - \eps$, and if $x \leq X$, then
\[ |(1 - x)^s\phi(x)| \leq (1 - x)^{a + 1 - \eps} \cdot C(1 - x)^{-a - 1} = C(1 - x)^{-\eps} .\]
Thus, $(1 - x)^s\phi(x) \to 0$ as $x \to -\infty$, so as long as $\rl(s) < a'$ as well,
\begin{align*}
\gent{\sigma(\phi)}(s) &= \int_{-\infty}^0 (1 - x)^s\phi'(x)\dx \\
&= \Big[(1 - x)^s\phi(x)\Big]_{-\infty}^0 + \int_{-\infty}^0 s(1 - x)^{s - 1}\phi(x)\dx \\
&= \phi(0) + s\gent{\phi}(s - 1) .
\end{align*}
\end{proof}

\begin{lemma} \label{gent and calS}
If $\phi \in \trable$, then for all $s,y \in \postt$,
\[ \gent{\calS^y(\phi)}(s) = \gent{\phi}(s + y) .\]
In the complex case, if $\phi \in \trable[a]$, then the same equation holds as long as $\rl(s) < a - \rl(y)$.
\end{lemma}

\begin{proof}
In the complex case, we have
\begin{align*}
\gent{\calS^y(\phi)}(s) = \int_{-\infty}^0 (1 - x)^s(1 - x)^y\phi(x)\dx &= \int_{-\infty}^0 (1 - x)^{s + y}\phi(x)\dx \\
&= \gent{\phi}(s + y) ,
\end{align*}
and the $p$-adic case was proven in Corollary~\ref{calS and shift for y in Zp}. One may also give a proof in the $p$-adic case that mimics the foregoing proof in the complex case, using Propositions~\ref{inttrans as an integral transform}, \ref{calS a Zp-action}, and \ref{calS acting via one - xtox}:
\begin{align*}
\gent{\calS^y(\phi)}(s) = \int_{\Zp} \Big((\one - \xtox)^{\myst s} \myst (\one - \xtox)^{\myst y} \myst \phi\Big)\dx[\dirac] &= \int_{\Zp} \big((\one - \xtox)^{\myst(s + y)} \myst \phi\big)\dx[\dirac] \\
&= \gent{\phi}(s + y) .
\end{align*}
\end{proof}

The function $\Phi_f$ satisfies the following functional equation.

\begin{theorem} \label{functional equation}
Let $f(t) = \sum_{k = 0}^n a_k t^k \in \Omega[t]$ with $n \geq 1$ and $a_n \not= 0$. Assume in the complex case that $(-1)^n\rl(a_n) < 0$ and in the $p$-adic case that $f(t) \in \pint\pwr{t}$, $f(0) \in \expdom$, and $f'(0) \in \punits$. Then $f(t)$ is compatible,
\beq \label{functional equation: prep}
\sigma(\phi_f) = \sum_{m = 0}^{n - 1} (-1)^m\left(\sum_{k = m + 1}^n k\ch{k - 1}{m}a_k\right)\calS^m(\phi_f) ,
\eeq
and for all $s \in \postt$,
\beq \label{functional equation: actual}
\exp(a_0) + s\Phi_f(s - 1) = \sum_{m = 0}^{n - 1} (-1)^m \left(\sum_{k = m + 1}^n k\ch{k - 1}{m}a_k\right)\Phi_f(s + m) .
\eeq
In both the complex case and the $p$-adic case, $\exp$ denotes the exponential power series.
\end{theorem}

\begin{proof}
The compatibility condition is met by definition in the $p$-adic case, and in the complex case we may apply Proposition~\ref{sufficiently general fn in trable}.

Proof of (\ref{functional equation: prep}). We begin with a calculation that will be used in both cases, the complex and the $p$-adic:
\begin{align}
f'(t) &= \sum_{k = 1}^n ka_k t^{k - 1} \nonumber \\
&= \sum_{k = 1}^n ka_k(1 - (1 - t))^{k - 1} \nonumber \\
&= \sum_{k = 1}^n ka_k \sum_{m = 0}^{k - 1} (-1)^m\ch{k - 1}{m}(1 - t)^m \nonumber \\
&= \sum_{k = 1}^n ka_k \sum_{m = 0}^{n - 1} (-1)^m\ch{k - 1}{m}(1 - t)^m \eqcom{since $\ch{k - 1}{m} = 0$ if $k \leq l \leq n - 1$} \nonumber \\
&= \sum_{m = 0}^{n - 1} (-1)^m \left(\sum_{k = m + 1}^n k\ch{k - 1}{m}a_k\right)(1 - t)^m . \label{func eqn prep}
\end{align}
Now, in the complex case, $\sigma(\phi_f)(x) = f'(x)\phi_f(x)$, so the foregoing calculation tells us that
\[ \sigma(\phi_f)(x) = \sum_{m = 0}^{n - 1} (-1)^m \left(\sum_{k = m + 1}^n k\ch{k - 1}{m}a_k\right)(1 - x)^m\phi_f(x) ,\]
and then the result follows by definition of $\calS^m$.

The argument in the $p$-adic case is similar but works at the level of exponential generating functions instead. Recall that, by definition, $\actcorr(\phi_f) = \gexp(f(t))$. Then
\begin{align*}
\actcorr(\sigma(\phi_f)) &= \actcorr(\phi_f)' \eqcom{by (\ref{shiftu and diff})} \\
&= \ddx[t]\Big(\mathopen{}\gexp(f(t))\Big) \\
&= f'(t)\gexp(f(t)) \\
&= f'(t)\actcorr(\phi_f) \\
&= \sum_{m = 0}^{n - 1} (-1)^m \left(\sum_{k = m + 1}^n k\ch{k - 1}{m}a_k\right)(1 - t)^m\actcorr(\phi_f) \eqcom{by (\ref{func eqn prep})} \\
&= \sum_{m = 0}^{n - 1} (-1)^m \left(\sum_{k = m + 1}^n k\ch{k - 1}{m}a_k\right)\actcorr(\calS^m(\phi_f)) \eqcom{by Prop.~\ref{(1 - t)^y}} \\
&= \actcorr\left(\sum_{m = 0}^{n - 1} (-1)^m \left(\sum_{k = m + 1}^n k\ch{k - 1}{m}a_k\right)\calS^m(\phi_f)\right) .
\end{align*}
The injectivity of $\actcorr$ now gives the desired result.

Proof of (\ref{functional equation: actual}). If $s \in \postt$,
\begin{align*}
\exp(a_0) &+ s\Phi_f(s - 1) \\*
&= \phi_f(0) + s\gent{\phi_f}(s - 1) \\
&= \gent{\sigma(\phi_f)}(s) \eqcom{by Lemma~\ref{gent and sigma}} \\
&= \sum_{m = 0}^{n - 1} (-1)^m\left(\sum_{k = m + 1}^n k\ch{k - 1}{m}a_k\right)\gent{\calS^m(\phi_f)}(s) \eqcom{by (\ref{functional equation: prep})} \\
&= \sum_{m = 0}^{n - 1} (-1)^m \left(\sum_{k = m + 1}^n k\ch{k - 1}{m}a_k\right)\gent{\phi_f}(s + m) \eqcom{by Lemma~\ref{gent and calS}} \\
&= \sum_{m = 0}^{n - 1} (-1)^m \left(\sum_{k = m + 1}^n k\ch{k - 1}{m}a_k\right)\Phi_f(s + m) .
\end{align*}
Note that in the complex case, it is legitimate to apply Lemma~\ref{gent and sigma}, because $\phi_f'$ is again in $\trable$ by Proposition~\ref{sufficiently general fn in trable}.
\end{proof}

\section{Interpolation} \label{sec: interpolation}

\subsection{Places}

Let $K$ be a number field. We recall some common notions. In the following, we exclude the trivial absolute value on $K$.

A \emph{place} of $K$, often denoted $\p$, is an equivalence class of absolute values, two absolute values being equivalent if they define the same topology on $K$.

An absolute value $|\cdot|$ is called \emph{non-archimedean} if the values $|n|$ are bounded as $n$ runs through the positive integers, and is called \emph{archimedean} otherwise. One may show that the non-archimedean ones are those satisfying the ultrametric property: $|x + y| \leq \max(|x|,|y|)$ for all $x,y \in K$. If one absolute value in a place is archimedean, then so are the rest, so we can speak of a place itself as being either archimedean or non-archimedean.

We refer the interested reader to \cite[Chp.~II, Sect.~3]{neukirch:alg} and \cite{cassels:global-fields} for further investigation of absolute values. Note that absolute values are called \emph{valuations} in \cite{neukirch:alg}, and that \cite{cassels:global-fields} uses a slightly more general notion of valuation.

\subsubsection{Canonical representatives of places}

It will be convenient to define the \emph{residue characteristic} of a place $\p$ of $K$, denoted $\resch{\p}$, to be the characteristic of the residue field of $\p$ in the non-archimedean case, and to be zero in the archimedean case.

We will assign to each place $\p$ of $K$ a canonical absolute value $|\cdot|_\p$ representing it. If $\p$ is archimedean, then we let $|\cdot|_\p$ be the unique absolute value of the form $x \mapsto |\tau(x)|_\C$ where $\tau : K \ra \C$ is an embedding. See \cite[Chp.~II, Thm.~4.2]{neukirch:alg}. If $K_\p$ is a completion of $K$ with respect to $|\cdot|_\p$, then of course the closure $\R_{K,\p}$ of $\Q$ in $K_\p$ is isomorphic to $\R$. We have a well-defined notion of the real part $\rl_\p(x)$ of an element $x \in K_\p$ with respect to $\p$, the real part being an element of $\R_{K,\p}$. If $\p$ is real, then $\rl_\p(x)$ is simply $x$. When $\p$ is complex, there is $\iota \in K_\p$ such that $\iota^2 = -1$ and $K_\p = \R_{K,\p}(\iota)$, and if $x = a + b\iota$ where $a,b \in \R_{K,\p}$, then its real part is defined to be $\rl_\p(x) = a$, which is independent of the choice of $\iota$.

Now we turn to non-archimedean places, which are in bijection with the non-zero prime ideals of the ring $\oh_K$ of algebraic integers of $K$. Let $\frakm_\p$ denote the prime ideal corresponding to $\p$. For a canonical representative of a non-archimedean place $\p$, we choose the unique absolute value $|\cdot|_\p$ satisfying
\[ |x|_\p = \Big(\tfrac{1}{p}\Big)^{\vp[\p](x)/e_\p} \]
for all $x \in K\st$, where $\vp[\p]$ is the $\frakm_\p$-adic valuation on $K$ and $e_\p$ is the ramification index of $\frakm_\p$ over $\Q$. If $p = \resch{\p}$, then $|p|_\p = \tfrac{1}{p}$.

In either case, the archimedean or the non-archimedean, we let $K_\p$ denote a completion of $K$ with respect to $|\cdot|_\p$. The valued field $(K_\p,|\cdot|_\p)$ is isomorphic, as a valued field, to
\begin{itemize}
\item $(\R,|\cdot|_\R)$ if $\p$ is real,
\item $(\C,|\cdot|_\C)$ if $\p$ is complex,
\item $(\Cp,|\cdot|_{\Cp})$ if $\p$ is non-archimedean, where $p = \resch{\p}$.
\end{itemize}

Other symbols to which we will attach a subscript $\p$ in order to show the dependence of its corresponding notion on $\p$ are
\[ \pint[\p] ,\, \maxid[\p] ,\, \units[\p] ,\, \punits[\p] ,\, \expdom[\p] ,\, \phi_{f,\p} ,\, \Phi_{f,\p} ,\, \pret_\p ,\, \postt_\p .\]
The first five have the obvious meanings in the non-archimedean case, and the rest have the obvious meanings in the general situation.

\subsubsection{The Teichm\"uller character}

If $\p$ is a non-archimedean place, let $\teich_\p : \units[\p] \ra \Omega_\p$ be the Teichm\"uller character as defined in \cite[Sect.~33]{schikhof:ultra-calc}, which is to say that for $x \in \units[\p]$, $\teich_\p(x)$ is the unique root of unity in $\Omega_\p$ of order prime to $\resch{\p}$ such that $|x - \teich_\p(x)|_\p < 1$. We then define, for each $x \in \units[\p]$,
\[ \princof{x} = \frac{x}{\teich_\p(x)} \in \punits[\p] .\]
If $x$ lies in a complete subfield $K$ of $\Omega_\p$, then $\teich_\p(x) \in K$, so $\princof{x} \in K$ as well.

If $\p$ is instead archimedean, we simply set $\princof{x} = x$ for each $x \in \Omega_\p\st$.

\brem
In the case $p = 2$, the following variant is often taken for the Teichm\"uller character on $\Z_2\st$, as in \cite[p.~51]{washington:cyc}. If $x \in \Z_2\st$, let $\widetilde{\teich}_2(x)$ be the unique element of $\{1,-1\}$ such that $x - \widetilde{\teich}_2(x) \in 4\Z_2$. Then one may set $\princof[2]{x}^\sim = x/\widetilde{\teich}_2(x) \in 1 + 4\Z$. Schikhof's version $\teich_2$ above and the modified version $\widetilde{\teich}_2$ are related by
\[ \widetilde{\teich}_2(x) = (-1)^x\teich_2(x) = (-1)^x ,\]
the expression $(-1)^x$ being well defined because $x \in \Z_2$, and then
\[ \princof[2]{x}^\sim = (-1)^x\princof[2]{x} = (-1)^x x .\]
The modified version is used, for example, in Iwasawa theory; see \cite[Theorem~5.11]{washington:cyc}. In what follows, one may use either version in the case $p = 2$, and we leave it to the reader to identify the minor changes required in Theorem~\ref{interpol of incomplete} if $\princof[2]{x}^\sim$ is desired instead of $\princof[2]{x}$. All other places, archimedean and non-dyadic non-archimedean, are unaffected.
\erem

\subsection{Examples}

\bex \label{linear example}
Let $K$ be a number field, and let $f(t) = a + rt$ where $a,r \in K$. Let $\p$ be a place of $K$, and to guarantee that $f(t)$ is compatible with $\p$, assume that
\begin{itemize}
\item if $\p$ is archimedean, then $\rl_\p(r) > 0$, and
\item if $\p$ is non-archimedean, then $a \in \expdom[\p]$ and $r \in \punits[\p]$.
\end{itemize}
By Theorem~\ref{functional equation},
\[ \Phi_{f,\p}(s) = \frac{1}{r}\big(s\Phi_{f,\p}(s - 1) + \exp(a)\big) \eqcom{for all $s \in \postt_\p$,} \]
so
\begin{equation} \label{linear f example: rec rel}
\left.
\begin{aligned}
\Phi_{f,\p}(0) &= \exp(a) ,\\
\Phi_{f,\p}(n) &= \frac{1}{r}\big(n\Phi_{f,\p}(n - 1) + \exp(a)\big) \eqcom{for all $n \in \Z_{\geq 1}$} . \quad
\end{aligned}
\right\}
\end{equation}
Note that $\exp(a) \in K_\p$.

Assume, first, that $a = 0$, so that $\exp(a) = 1 \in K$. Then (\ref{linear f example: rec rel}) shows that the sequence $(\Phi_{f,\p}(n))_{n \geq 0}$ is contained in $K$ and is independent of $\p$. Thus, if $\p$ and $\q$ are two places meeting the hypotheses above (one may be archimedean and the other non-archimedean, or any combination), and if $\Phi_{f,\p} : \postt_\p \ra \Omega_\p$ and $\Phi_{f,\q} : \postt_\q \ra \Omega_\q$ denote the corresponding functions, then
\[ \Phi_{f,\p}(n) = \Phi_{f,\q}(n) \eqcom{for all $n \in \Z_{\geq 0}$.} \]
This equality makes sense, because both sides are in $K$.

Next, assume that $a \not= 0$. Then in the archimedean case, $\exp(a)$ is transcendental by the Lindemann--Weierstrass Theorem (see, for example, \cite[Chap.~4]{murty-rath:trnum}), and in the non-archimedean case, transcendence is proven in \cite{adams:transc-p-adic}; see the last paragraph of Section~3 in that paper. Therefore, again taking two places $\p$ and $\q$ as above, if we let $\alpha_\p = \exp(a)$ where the limit is taken in $K_\p$, and let $\alpha_\q = \exp(a)$ in $K_\q$ instead, then there is a unique isomorphism $\rho_{\p,\q} : K(\alpha_\p) \ra K(\alpha_\q)$ over $K$ mapping $\alpha_\p$ to $\alpha_\q$. Because $\rho_{\p,\q}(r) = r$, applying $\rho_{\p,\q}$ to the version of (\ref{linear f example: rec rel}) for $\p$ results in the version for $\q$, so the sequences $(\rho_{\p,\q}(\Phi_{f,\p}(n)))_{n \geq 0}$ and $(\Phi_{f,\q}(n))_{n \geq 0}$ in $K_\q$ are equal. Thus, via $\rho_{\p,\q}$, the functions $\Phi_{f,\p}$ and $\Phi_{f,\q}$ interpolate each other on $\Z_{\geq 0}$.

Taking $a = 0$ leads to interpolation of $\Gamma^\theta(\cdot,r)$ in the limited setting where $r$ is positive at some archimedean place and $r$ is a principal unit at some non-archimedean place. We omit the details because we will generalize this in Section~\ref{sec: interp of inc gamma}.
\eex

\bex
Let $K$ be a number field, and let $a,b,c \in K$. Suppose that $\p$ is a non-archimedean place of $K$ such that $b/2$ and $c/3$ are $\p$-integral and $a + b + c$ is a $\p$-principal unit, meaning that $\vp[\p](a + b + c - 1) > 0$. Then we may construct a cubic polynomial $f(t) \in K[t]$ such that
\beq \label{cubic example: func eqn}
1 + s\Phi_{f,\p}(s - 1) = a\Phi_{f,\p}(s) + b\Phi_{f,\p}(s + 1) + c\Phi_{f,\p}(s + 2) \eqcom{for all $s \in \Zp$,}
\eeq
where $p = \resch{\p}$, thus obtaining a locally analytic $\p$-adic solution to the recurrence relation
\[ 1 + nx_{n - 1} = ax_n + bx_{n + 1} + cx_{n + 2} .\]
To construct $f(t)$, we observe that the unique solution $(a',b',c') \in K^3$ to the system
\[ a' + 2b' + 3c' = a ,\quad 2b' + 6c' = -b ,\quad 3c' = c \]
is
\[ a' = a + b + c ,\quad b' = -\frac{1}{2}b - c ,\quad c' = \frac{1}{3}c ,\]
all of which elements are $\p$-integral by assumption, and further $a'$ is $\p$-principal. Therefore, the polynomial $f(t) = a't + b't^2 + c't^3$ is compatible with the non-archimedean place $\p$, and if, in addition, $\q$ is an archimedean place such that $\rl_\q(c) > 0$, then $f(t)$ is compatible with $\q$ as well. Thus, by Theorem~\ref{functional equation}, (\ref{cubic example: func eqn}) holds not only as stated but also if $\p$ is replaced with $\q$ and $s$ is replaced by an element of $\Omega_\q$ (a copy of $\C$). If we restrict to the situation $a,b,c \in \oh_K$, $a + b + c = 1$, and $c$ totally positive, then we obtain a solution to the functional equation (\ref{cubic example: func eqn}) for every place $\p$ that either is archimedean or is non-archimedean with residue characteristic greater than $3$.

Whether the various solutions $\Phi_{f,\p}$ can be thought of as interpolating one another is not as straightforward as in the simple situation of Example~\ref{linear example}. Restricting $s$ to be a non-negative integer $m$, we obtain a recurrence relation
\[ \Phi_{f,\p}(m + 2) = \frac{1}{c}\Big(1 + m\Phi_{f,\p}(m - 1) - a\Phi_{f,\p}(m) - b\Phi_{f,\p}(m + 1)\Big) ,\]
from which it follows that the values at all non-negative integers are determined by the two values $\Phi_{f,\p}(0)$ and $\Phi_{f,\p}(1)$. If these were known to be algebraically independent for some place $\p$ and, separately, algebraically independent for another place $\q$, then we could provide a field isomorphism $K_\p \ra K_\q$ over $K$ mapping $\Phi_{f,\p}(0)$ to $\Phi_{f,\q}(0)$ and, likewise, $\Phi_{f,\p}(1)$ to $\Phi_{f,\q}(1)$, and then $\Phi_{f,\p}(m)$ would map to $\Phi_{f,\q}(m)$ for all $m \in \Z_{\geq 0}$ because of the recurrence relation. However, such algebraic independence is not currently known, as far as the author is aware, although there is no reason to doubt it either.
\eex

\subsection{$p$-adic interpolation of incomplete $\Gamma$-functions} \label{sec: interp of inc gamma}

Let $\theta \in \R$. We fix $r \in \Q\st$ and turn to $p$-adic interpolation of $\Gamma^\theta(\cdot,r)$ for primes $p$ satisfying only the assumption that $\vp(r) = 0$. This means that for a given $r$, we move from being able to $p$-adically interpolate $\Gamma^\theta(\cdot,r)$ for only finitely many $p$ (finite if $r \not= 1$, that is) to being able to do so for all \emph{except} finitely many $p$.

For consistency with previous sections, we work with places $\p$ rather than primes $p$, even though our attention will be on the field of rationals.

For our chosen $r \in \Q\st$, let $P_r$ be the finite set of non-archimedean places $\p$ of $\Q$ for which $\vp[\p](r) \not= 0$, and define a power series $f_r(t) \in \Q\pwr{t}$ by $f_r(t) = -r((1 - t)^{\tfrac{1}{r}} - 1)$, which is simply to say
\[ f_r(t) = -r\sum_{k = 1}^\infty (-1)^k\ch{1/r}{k}t^k .\]
An important property of the formal power series $f_r(t)$ is the fact, arising from (\ref{props of a(t)}), that
\[ f_r'(t) = (1 - t)^{\tfrac{1}{r} - 1} .\]

We will need to single out the case where the conditions $\p$ archimedean and $r < 0$ hold simultaneously. This will be called the \emph{exceptional case}. The complementary case, \ie where either (i) $\p$ is non-archimedean or (ii) $\p$ is archimedean and $r > 0$, will be called the \emph{ordinary case}.

\begin{lemma} \label{compatibility lemma}
If $\p$ is a place of $\Q$, archimedean or otherwise, that is not in $P_r$, then in the terminology of Section~\ref{sec: compatible}, the power series $f_r(t)$ is compatible with $\p$ in the ordinary case, and is $(-1)$-compatible with $\p$ in the exceptional case.
\end{lemma}

\begin{proof}
If $\p$ is non-archimedean, $f_r(t)$ meets the compatibility requirements as follows:
\begin{itemize}
\item $f_r(t) \in \pint[\p]\pwr{t}$ because the assumption $\vp[\p](r) = 0$ ensures that $-r(-1)^k\ch{1/r}{k} \in \pint[\p]$ for all $k$.
\item $f(0) = 0 \in \expdom[\p]$.
\item $f'(0) = 1 \in \punits[\p]$.
\end{itemize}

Now we suppose that $\p$ is archimedean. The power series $f_r(t)$ is a polynomial if $1/r \in \Z_{\geq 1}$, and otherwise it is a convergent power series whose radius of convergence in $\C$ is $1$ by the ratio test. We may therefore define an analytic function
\defmap{h}{\mathscr{D}}{\C}{x}{-r\sum_{k = 1}^\infty (-1)^k\ch{1/r}{k}t^k}
where $\mathscr{D}$ is the open unit disc in $\C$ centred at the origin.

For $z$ not in the ray $R_\pi$, define complex powers $z^s$ via $\log_\pi$, the principal branch of $\log$; see the notation of Section~\ref{sec: cmx inc gamma}. Letting $V = \C \setm \R_{\geq 1}$, we may consequently define an analytic function
\defmap{\widehat{f}}{V}{\C}{x}{-r((1 - x)^{\tfrac{1}{r}} - 1)}[.]
We claim that $\widehat{f}$ agrees with the function $h$ on $\mathscr{D}$. Formal differentiation of the power series defining $h$ shows, via (\ref{props of a(t)}), that
\[ h'(x) = \frac{1}{1 - x}(1 - x)^{\tfrac{1}{r}} = \frac{1}{1 - x}\left(1 - \frac{1}{r}h(x)\right) ,\]
but differentiation of $x \mapsto \exp(\tfrac{1}{r}\log(1 - x))$ shows that, likewise,
\[ \widehat{f}'(x) = \frac{1}{1 - x}\left(1 - \frac{1}{r}\widehat{f}(x)\right) .\]
Thus, $h$ and $\widehat{f}$ satisfy the same first-order differential equation in $\mathscr{D}$, and they also both vanish at the origin, so they define the same function in that disc.

It remains to show that the function $\R_{\leq 0} \ra \C$, $x \mapsto \exp(\widehat{f}(x))$, is in $\trable$ if $r > 0$ and is in $\trable[-1]$ if $r < 0$. In the positive case, we achieve this by showing that $\widehat{f}$ meets the hypotheses of the function called $F$ in Proposition~\ref{sufficiently general fn in trable}. If $\nu \in (0,\tfrac{1}{r})$ and $x < 0$, then
\[ \frac{(-x)^\nu}{(1 - x)^{\tfrac{1}{r}}} = (-x)^{\nu - \tfrac{1}{r}}\left(\frac{-x}{1 - x}\right)^{\tfrac{1}{r}} ,\]
and this tends to $0$ as $x \to -\infty$ because $\nu - \tfrac{1}{r} < 0$. Therefore, because $r > 0$, there is $X \leq -1$ such that for all $x \leq X$, $(-x)^\nu \leq r(1 - x)^{\tfrac{1}{r}}$, \ie
\[ -r(1 - x)^{\tfrac{1}{r}} \leq -(-x)^\nu .\]
As necessary, we may now make $X$ more negative to ensure that
\[ -r((1 - x)^{\tfrac{1}{r}} - 1) \leq -(-x)^\nu \eqcom{for all $x \leq X$.} \]

If $r < 0$, then
\[ |\mathopen{}\exp(\widehat{f}(x))| = \exp\mathopen{}\Big(\mathopen{}\rl\mathopen{}\big({-r}((1 - x)^{\tfrac{1}{r}} - 1)\big)\Big) \leq C(1 - x)^0 \]
for some constant $C > 0$, so indeed the map $x \mapsto \exp(\widehat{f}(x))$ is in $\trable[-1]$.
\end{proof}

In light of Lemma~\ref{compatibility lemma}, $\Phi_{f_r,\p}$ is defined when $\p \not\in P_r$, and we set
\defmap{\Psi_{r,\p}}{\postt_\p}{\Omega_\p}{s}{\princof{r}^s\Phi_{f_r,\p}(\tfrac{1}{r}(s + 1) - 1)}
in the ordinary case, and
\defmap{\Psi_{r,\p}}{\C_{> -1}}{\Omega_\p}{s}{r^s\Phi_{f_r,\p}(\tfrac{1}{r}(s + 1) - 1) + \exp(r)\gammahat(s + 1)}
in the exceptional case.

\begin{prop} \label{f_r-g_r switch}
If $\p$ is the archimedean place, then $\Psi_{r,\p} = \gfn^\theta(\cdot,r)$ as functions, as long as in the case where $r > 0$ we assume that $\sin(\theta) \geq 0$ if $\cos(\theta) > 0$.
\end{prop}

\begin{proof}
In the integral defining $\Phi_{f_r,\p}(\tfrac{1}{r}(s + 1) - 1)$, we will make the change of variables $y = 1 - (1 - x)^{1/r}$, so that
\[ r\dx[y] = (1 - x)^{\tfrac{1}{r} - 1}\dx \quad\text{and}\quad (1 - x)^{\tfrac{1}{r}} = 1 - y .\]
If $r > 0$, then for all $s \in \C$,
\begin{align*}
\Psi_{r,\p}(s) &= r^s\Phi_{f_r,\p}(\tfrac{1}{r}(s + 1) - 1) \\
&= r^s\int_{-\infty}^0 (1 - x)^{\tfrac{1}{r}(s + 1) - 1}\exp\mathopen{}\Big({-r}((1 - x)^{\tfrac{1}{r}} - 1)\Big)\dx \\
&= r^s\int_{-\infty}^0 (1 - x)^{\tfrac{s}{r}}\exp\mathopen{}\Big({-r}((1 - x)^{\tfrac{1}{r}} - 1)\Big)(1 - x)^{\tfrac{1}{r} - 1}\dx \\
&= r^s\int_{-\infty}^0 (1 - y)^s\exp(ry) \cdot r\dx[y] \\
&= \gfn^\theta(s,r)
\end{align*}
by Proposition~\ref{standardized form of gfn}\ref{standardized form: gfn}.

We use the same formula for the change of variables in the case $r < 0$, but the lower limit of $-\infty$ in the integral becomes $1$ instead, that is,
\begin{align*}
\Psi_{r,\p}(s) &= r^s\int_1^0 (1 - y)^s\exp(ry) \cdot r\dx[y]  + \exp(r)\gammahat(s + 1) \\
&= -\lgfn^\theta(r,s) + \exp(r)\gammahat(s + 1) \eqcom{by Proposition~\ref{standardized form of gfn}\ref{standardized form: lgfn}} \\
&= \gfn^\theta(r,s) \eqcom{by (\ref{Gamma as sum}).}
\end{align*}
\end{proof}

\begin{lemma} \label{func eqn in prep for incomplete gamma}
The equality
\[ \Phi_{f_r,\p}(s + \tfrac{1}{r} - 1) = 1 + s\Phi_{f_r,\p}(s - 1) \]
holds for all $s \in C_\p$ in the ordinary case and holds for all $s \in \C_{< 0}$ in the exceptional case.
\end{lemma}

\begin{proof}
For the purposes of the proof, we drop the subscript $\p$ on our objects. First, we show that $\sigma(\phi_{f_r}) = \calS^{\tfrac{1}{r} - 1}(\phi_{f_r})$, which in both cases uses the observation that $f_r'(t) = (1 - t)^{\tfrac{1}{r} - 1}$. If $\p$ is archimedean,
\[ \sigma(\phi_{f_r})(x) = \phi_{f_r}'(x) = f_r'(x)\phi_{f_r}(x) = (1 - x)^{\tfrac{1}{r} - 1}\phi_{f_r}(x) = \calS^{\tfrac{1}{r} - 1}(\phi_{f_r})(x) ,\]
and if $\p$ is non-archimedean,
\begin{align*}
\actcorr(\sigma(\phi_{f_r})) &= \actcorr(\phi_{f_r})' \\
&= \ddx[t]\Big(\mathopen{}\exp(f_r(t))\Big) \\
&= f'(t)\exp(f_r(t)) \\
&= (1 - t)^{\tfrac{1}{r} - 1}\actcorr(\phi_{f_r}) \\
&= \actcorr\Big(\calS^{\tfrac{1}{r} - 1}(\phi_{f_r})\Big) \eqcom{by Proposition~\ref{(1 - t)^y},}
\end{align*}
so $\sigma(\phi_{f_r}) = \calS^{\tfrac{1}{r} - 1}(\phi_{f_r})$ by the injectivity of $\actcorr$.

Hence, in either case,
\begin{align*}
\Phi_{f_r}(s + \tfrac{1}{r} - 1) &= \gent{\phi_{f_r}}(s + \tfrac{1}{r} - 1) \\
&= \gent{\calS^{\tfrac{1}{r} - 1}(\phi_{f_r})}(s) \eqcom{by Lemma~\ref{gent and calS}} \\
&= \gent{\sigma(\phi_{f_r})}(s) \\
&= \phi_{f_r}(0) + s\gent{\phi_{f_r}}(s - 1) \eqcom{by Lemma~\ref{gent and sigma}} \\
&= 1 + s\Phi_{f_r}(s - 1) .
\end{align*}
We remark on why we may apply the lemmas in the archimedean case. If $r > 0$, then $\phi_f' \in \trable$ by Proposition~\ref{sufficiently general fn in trable}, so both lemmas hold. If $r < 0$, then $\phi_{f_r} \in \trable[-1]$, and $\phi_{f_r}'$, which satisfies $\phi'(x) = (1 - x)^{\tfrac{1}{r} - 1}\phi(x)$, is consequently in $\trable[0]$. Therefore, in the notation of Lemma~\ref{gent and sigma}, we may take $a = -1$ and $a' = 0$, so we require only $\rl(s) < 0$ in that lemma. The hypotheses of the other lemma are even easier to verify.
\end{proof}

\begin{theorem} \label{interpol of incomplete}
Let $r \in \Q\st$, and let $\p,\q$ be places of $\Q$ not in the finite set $P_r$.
\begin{enumerate}[label=(\roman*)]
\item $\princof{r}^{-m}\Psi_{r,\p}(m) = \princof[\q]{r}^{-m}\Psi_{r,\q}(m)$ for all $m \in \Z_{\geq 0}$, this being an equality of elements of $\Q$.\label{interpol of incomplete: all m}
\item $\Psi_{r,\p}(m) = \Psi_{r,\q}(m)$ for all $m \in \Z_{\geq 0}$ divisible by $\lcm(\resch{\p} - 1,\resch{\q} - 1)$.\label{interpol of incomplete: restricted m}
\end{enumerate}
\end{theorem}

\begin{proof}
\ref{interpol of incomplete: all m} It will be convenient to define $\widetilde{\Psi}_\p : \postt_\p \ra \Omega_\p$ in the ordinary case and $\widetilde{\Psi}_\p : \C_{> -1} \ra \C$ in the exceptional case by
\[ \widetilde{\Psi}_\p : s \mapsto \princof{r}^{-s}\Psi_{r,\p}(s) .\]
In the ordinary case,
\begin{align*}
\widetilde{\Psi}_\p(s) = \Phi_{f_r,\p}(\tfrac{s}{r} + \tfrac{1}{r} - 1) &= 1 + \frac{s}{r}\Phi_{f_r,\p}(\tfrac{s}{r} - 1) \eqcom{by Lemma~\ref{func eqn in prep for incomplete gamma}} \\
&= 1 + \frac{s}{r}\widetilde{\Psi}_\p(s - 1) .
\end{align*}
Therefore,
\[ \widetilde{\Psi}_\p(0) = 1 \quad\text{and}\quad \widetilde{\Psi}_\p(m) = 1 + \frac{m}{r}\widetilde{\Psi}_\p(m - 1) \]
for all $m \in \Z_{\geq 1}$.

We may obtain the same initial value and recurrence relation in the exceptional case, this time using the fact that
\[ \widetilde{\Psi}_\p(s) = \Phi_{f_r,\p}(\tfrac{1}{r}(s + 1) - 1) + r^{-s}\exp(r)\gammahat(s + 1) .\]
Indeed,
\begin{align*}
\widetilde{\Psi}_\p(0) &= \Phi_{f_r}(\tfrac{1}{r} - 1) + \exp(r) \eqcomb{Proposition~\ref{rec rels for both gammas}, penultimate equation} \\
&= \int_{-\infty}^0 (1 - x)^{\tfrac{1}{r} - 1}\exp\mathopen{}\Big({-r}((1 - x)^{\tfrac{1}{r}} - 1)\Big)\dx + \exp(r) \\
&= \left[\mathopen{}\exp\mathopen{}\Big({-r}((1 - x)^{\tfrac{1}{r}} - 1)\Big)\right]_{-\infty}^0 + \exp(r) \\
&= 1 \eqcom{because $r < 0$,}
\end{align*}
and if $m \in \Z_{\geq 1}$,
\begin{align*}
\widetilde{\Psi}_\p(m) &= \Phi_{f_r}(\tfrac{m}{r} + \tfrac{1}{r} - 1) + r^{-m}\exp(r)\gammahat(m + 1) \\
&= 1 + \frac{m}{r}\Phi_{f_r}(\tfrac{m}{r} - 1) + r^{-m}\exp(r)\gammahat(m + 1) \eqcom{by Lemma~\ref{func eqn in prep for incomplete gamma} ($\tfrac{m}{r} < 0$)} \\
&= 1 + \frac{m}{r}\Big(\widetilde{\Psi}_\p(m - 1) - r^{-m + 1}\exp(r)\gammahat(m)\Big) + r^{-m}\exp(r)\gammahat(m + 1) \\
&= 1 + \frac{m}{r}\widetilde{\Psi}_\p(m - 1) + r^{-m}\exp(r)(\gammahat(m + 1) - m\gammahat(m)) \\
&= 1 + \frac{m}{r}\widetilde{\Psi}_\p(m - 1) \eqcomb{Proposition~\ref{rec rels for both gammas}, last equation}[.]
\end{align*}

In either case, the ordinary or the exceptional, we see from the initial value and recurrence relation that the sequence $(\widetilde{\Psi}_\p(m))_{m \geq 0}$ is a sequence of rational numbers that is independent of $\p$.

\ref{interpol of incomplete: restricted m} This part follows immediately from part~\ref{interpol of incomplete: all m} in light of the observation that $\princof{r}^m = r^m$ if $\resch{\p} - 1 \dv m$.
\end{proof}

We finally give a definition of the incomplete gamma-function in the $p$-adic situation, and state and prove the interpolation property. Define
\[ \tilde{p} =
\begin{cases}
p & \text{if $p > 2$,} \\
4 & \text{if $p = 2$.}
\end{cases}
\]
The $p$-adic exponential, if considered as a group homomorphism $\tilde{p}\Zp \ra 1 + \maxid[p]$, has $\tilde{p}$ extensions to group homomorphisms $\extexp : \Zp \ra 1 + \maxid[p]$. Indeed, these extensions are in bijection with the $\tilde{p}$th roots of $\exp(\tilde{p})$, with an extension $\extexp$ corresponding to $\extexp(1)$.

Choose, then, any of these $\tilde{p}$ extensions, and call it $\extexp$. It is locally analytic because of the fact that
\[ \extexp(x) = \extexp(x_0 + (x - x_0)) = \extexp(x_0)\extexp(x - x_0) = \extexp(x_0)\exp(x - x_0) \]
if $|x - x_0| < \oop^{\tfrac{1}{p - 1}}$. Further, if $\alpha \in \Zp \setm \{0\}$ is algebraic over $\Q$, then $\extexp(\alpha)$ is transcendental. Indeed, $\extexp(\alpha)^{\tilde{p}} = \extexp(\tilde{p}\alpha) = \exp(\tilde{p}\alpha)$, and this is transcendental by the result of Adams~\cite{adams:transc-p-adic} used in Example~\ref{linear example}. (Actually, the local analyticity and the transcendence property both hold more generally: Any group-theoretic extension of $\exp$ to $\Cp \ra 1 + \maxid[p]$ is locally analytic and maps non-zero algebraic elements to transcendental elements.)

If $r \in \Q\st$, $\p$ is a non-archimedean place of $\Q$ such that $\vp[\p](r) = 0$, and $p = \resch{\p}$, define
\defmap{\Gamma_\p^\extexp(\cdot,r)}{\Zp}{\Cp}{s}{\extexp(-r)\Psi_{r,\p}(s - 1)}[,]
a locally analytic function because $\Psi_{r,\p}$ is by the general theory in Section~\ref{sec: 2-var}. In terms of $p$-adic integration,
\begin{align}
\Gamma_\p^\extexp(s,r) &= \extexp(-r)\princof{r}^{s - 1}\int_{\Zp} \Big((\one - \xtox)^{\myst\big(\tfrac{s}{r} - 1\big)} \myst \phi_{f_r}\Big)\dx[\dirac] \label{p-adic inc gamma via integral} \\
&= \extexp(-r)\princof{r}^{s - 1}\int_{\Zp} (\one - \xtox)^{\myst\big(\tfrac{s}{r} - 1\big)}\dx[\mu_r] \nonumber
\end{align}
for all $s \in \Zp$, where $\mu_r$ is the measure $\mu_{\phi,-1}$ in the case $\phi = \phi_{f_r}$; see Section~\ref{sec: p-adic integrals} and Proposition~\ref{inttrans as an integral transform}. Compare (\ref{p-adic inc gamma via integral}) with the complex version,
\[ \Gamma^\theta(s,r) = \exp(-r)r^{s - 1}\int_{-\infty}^0 (1 - x)^{\tfrac{s}{r} - 1}\phi_{f_r}(x)\dx ,\]
valid in the case $r > 0$ as long as $\sin(\theta) \geq 0$ if $\cos(\theta) > 0$.

\begin{cor}[$p$-adic interpolation of incomplete gamma-functions]
Let $r \in \Q\st$, and let $\p,\q$ be places of $\Q$ lying outside the finite set $P_r$. If $\p$ is archimedean, let $\extexp_\p$ be a choice of $\theta \in \R$, and if $\p$ is non-archimedean, let $\extexp_\p$ be a choice of group-theoretic extension $\Zp \ra 1 + \maxid[p]$ of $\exp$, and choose $\extexp_\q$ similarly.
\begin{enumerate}[label=(\roman*)]
\item The values $\Gamma_\p^{\extexp_\p}(m,r)$ with $m \in \Z_{\geq 1}$ are all in the field $\Q(\Gamma_\p^{\extexp_\p}(1,r))$, and similarly for $\q$.
\item There is a unique field isomorphism $\rho_{\p,\q} : \Q(\Gamma_\p^{\extexp_\p}(1,r)) \ra \Q(\Gamma_\q^{\extexp_\q}(1,r))$ mapping $\Gamma_\p^{\extexp_\p}(1,r)$ to $\Gamma_\q^{\extexp_\q}(1,r)$.
\item For all $m \in \Z_{\geq 1}$ congruent to $1$ modulo $\lcm(\resch{\p} - 1,\resch{\q} - 1)$,
\[ \rho_{\p,\q}(\Gamma_\p^{\extexp_\p}(m,r)) = \Gamma_\q^{\extexp_\q}(m,r) ,\]
and if $\p$ is non-archimedean and $\q$ archimedean, then the set of such $m$ is simply $1 + (p - 1)\Z_{\geq 0}$, where $p = \resch{\p}$, and is therefore dense in $\Zp$.
\end{enumerate}
\end{cor}

\begin{proof}
The first fact is a simple consequence of the recurrence relation satisfied by the sequence $(\Gamma_\p^{\extexp_\p}(m,r))_{m \geq 1}$. The second is a result of the transcendence of $\Gamma_\p^{\extexp_\p}(1,r)$. In the archimedean case, $\Gamma_\p^{\extexp_\p}(1,r) = \exp(-r)$, transcendental by the Lindemann--Weierstrass Theorem, and in the non-archimedean case, $\Gamma_\p^{\extexp_\p}(1,r) = \extexp(-r)$, transcendental by the comment above.

For the interpolation property, we may assume that $\theta = \pi$, because the values $\Gamma^\theta(m,r)$ with $m \in \Z$ are independent of $\theta$ by Remark~\ref{rem: independent of theta}. We point out that
\[ \Gamma_\p^{\extexp_\p}(s,r) = \Gamma_\p^{\extexp_\p}(1,r)\Psi_{r,\p}(s - 1) \]
for all $s$ in the domain of $\Gamma_\p^{\extexp_\p}(\cdot,r)$. In the archimedean case, this uses Propositions~\ref{rec rels for both gammas} and \ref{f_r-g_r switch}, and it is almost a definition in the non-archimedean case. Hence, if $m$ is a positive integer such that $m - 1$ is divisible by $\lcm(\resch{\p} - 1,\resch{\q} - 1)$, then
\begin{align*}
\rho_{\p,\q}(\Gamma_\p^{\extexp_\p}(m,r)) &= \rho_{\p,\q}\Big(\Gamma_\p^{\extexp_\p}(1,r)\Psi_{r,\p}(m - 1)\Big) \\
&= \rho_{\p,\q}(\Gamma_\p^{\extexp_\p}(1,r))\Psi_{r,\p}(m - 1) \eqcom{because $\Psi_{r,\p}(m - 1) \in \Q$} \\
&= \Gamma_\q^{\extexp_\q}(1,r)\Psi_{r,\p}(m - 1) \eqcom{by definition of $\rho_{\p,\q}$} \\
&= \Gamma_\q^{\extexp_\q}(1,r)\Psi_{r,\q}(m - 1) \eqcom{by the theorem} \\
&= \Gamma_\q^{\extexp_\q}(m,r) .
\end{align*}
The very final assertion, regarding the case where $\q$ is archimedean, is clear.
\end{proof}

\end{document}